\documentclass{article}
\usepackage{amsmath}
\usepackage{amssymb}
\usepackage[all]{xy}
\usepackage{enumitem}
\usepackage{amsthm}
\usepackage[utf8]{inputenc}
\usepackage[english]{babel}
\usepackage{fancyhdr}
\usepackage{graphicx}
\newcommand{\N}{\ensuremath{\mathbb{N}}}
\newcommand{\Z}{\ensuremath{\mathbb{Z}}}
\newcommand{\Q}{\ensuremath{\mathbb{Q}}}
\newcommand{\C}{\ensuremath{\mathbb{C}}}
\newcommand{\X}{\ensuremath{\mathbb{X}}}

\newcommand*{\defeq}{\mathrel{\vcenter{\baselineskip0.5ex \lineskiplimit0pt
			\hbox{\scriptsize.}\hbox{\scriptsize.}}}%
	=}

\theoremstyle{plain}
\newtheorem{theorem}{Theorem}[section]
\newtheorem{prop}[theorem]{Proposition}
\newtheorem{defin}[theorem]{Definition}
\newtheorem{lemma}[theorem]{Lemma}
\newtheorem{corol}[theorem]{Corollary}
\newtheorem{rem}[theorem]{Remark}
\newtheorem{conj}[theorem]{Conjecture}

\numberwithin{equation}{subsection}

\title{Periodic self maps and thick ideals in the stable motivic homotopy category over \C\ at odd primes}
\author{Sven-Torben Stahn}
\begin{document}
\maketitle 
\begin{abstract}In this article we study thick ideals defined by periodic self maps in the stable motivic homotopy category over \C. In addition, we extend some results of Ruth Joachimi about the relation between thick ideals defined by motivic Morava K-theories and the preimages of the thick ideals in the stable homotopy category under Betti realization.
\end{abstract}
\setcounter{tocdepth}{2}
\tableofcontents 

\section{Introduction}
There are two famous results by Hopkins and Smith in \cite{HS} that provide a complete description of the thick subcategories in the stable homotopy category of finite topological spectra.
\begin{defin}
	A thick subcategory of a tensor triangulated category is a nonempty, full, triangulated subcategory that is closed under retracts. A thick ideal is a thick subcategory that is closed under tensoring with arbitrary objects.
\end{defin}
The thick subcategory theorem states that if we localize at a prime \(l\) the thick subcategories (in fact thick ideals) of the category \(\mathcal{SH}_{(l)}^{fin}\) are given by a chain 
\[\mathcal{SH}_{(l)}^{fin}=\mathcal{C}_0 \supsetneq \mathcal{C}_1 \supsetneq \mathcal{C}_2 \supsetneq ... \supsetneq \mathcal{C}_\infty=\{0\}\]
and each thick ideal \(\mathcal{C}_{i+1}, 0\leq i < \infty\), is characterized by the vanishing of the \(i\)-th Morava K-theory \(K(i)\), where \(K(0)=H\Q\) by convention. The periodicity theorem states that these thick subcategories can also be described by the property of admitting a special kind of periodic self map; a so called \(v_n\)-self map that induces an isomorphism in \(K(n)\) and nilpotent maps in \(K(m),\ m\neq n\). Using the older Nilpotence theorem of Devinatz, Hopkins and Smith in \cite{DHS}, Hopkins and Smith showed that the full subcategory \(\mathcal{C}_{v_n}\) of finite spectra admitting such self maps is in fact thick, and thus equal to one of the categories \(\mathcal{C}_i\). For algebraic reasons (see \cite[3.3.11]{RAV2}) the category \(\mathcal{C}_{v_n}\) must be nested in the following way: \[\mathcal{C}_{n+1}\subset\mathcal{C}_{v_n}\subset\mathcal{C}_{n}\]
Therefore, by the thick subcategory theorem, the existence of at least one spectrum \(X_n\) in \(\mathcal{C}_{n}\) admitting such a self map proves the equality \(\mathcal{C}_{v_n}=\mathcal{C}_{n}\). Using an earlier construction of Smith, they prove that there is indeed such a spectrum \(X_n\) that admits a \(v_n\)-self map.\\

The fact that the motivic Hopf map \(\eta\) is not nilpotent suggests that the picture looks very different in the motivic context, even over the complex numbers. Ruth Joachimi showed in her dissertation that algebraic Morava K-theories, originally defined by Borghesi, define a similar chain of thick subcategories \(\mathcal{C}_{AK(n)}\) for odd primes over the base field \C (\cite[9.6.4]{JOA}), but also that there are a more thick ideals in the motivic homotopy category (\cite[Chapter 7]{JOA}). In addition, she relates the thick ideals \(\mathcal{C}_{AK(n)}\) to the thick ideals \(\text{thickid}(c\mathcal{C}_n)\) and \(R^{-1}(\mathcal{C}_n)\) provided by the classical thick ideals via the constant simplicial presheaf and Betti realization functors, respectively.\\

The purpose of this article is to explore the motivic equivalents of the constructions by Hopkins and Smith.
In Theorem \ref{ThickSubcat}, we prove that periodic motivic self maps defined by algebraic Morava K-theory define a thick subcategory, but we need to make use of a conjectural weakened version of a motivic nilpotence lemma. In Theorem \ref{SelfMapExample} we lift a construction by Hopkins and Smith in \cite{HS} to the motivic world to show that examples of these self maps exist. Finally, in the last two sections, we use some of our computations in the preceding sections to settle \cite[Conjecture 7.1.7.3]{JOA}. We furthermore provide a counterexample to the asserted inclusion \(\text{thickid}(c\mathcal{C}_2) \subset \mathcal{C}_{AK(1)}\) in \cite[Chapter 9, last section]{JOA} and we identify an error in \cite[Proposition 8.7.3]{JOA}, on which the assertion is based. The counterexample also proves that the inclusion \(\mathcal{C}_{AK(1)}\subset R^{-1}(\mathcal{C}_2)\) is actually proper and hence that the thick subcategories defined by algebraic Morava K-theories are distinct from the preimages of the topological thick ideals under Betti realization.\\

This research was originally part of my dissertation under supervision by Jens Hornbostel, to who I am grateful for his support. It was conducted in the framework of the research training group
\emph{GRK 2240: Algebro-Geometric Methods in Algebra, Arithmetic and Topology},
which is funded by the DFG.

\section{Background}
We work in the motivic stable homotopy category \(\mathcal{SH}_\C\), whose objects are \(\mathbb{P}^1\)-spectra of motivic spaces over the base field \(\C\). The construction of this category is due to Voevodsky and Morel (see \cite{VOE} and \cite{MV}) and mimicks the construction  of the topological stable homotopy category, where smooth schemes take the place of topological spaces. There are two kinds of spheres in the motivic world, a simplicial and a geometric one; therefore suspensions, homotopy, homology and cohomology  are all not singly graded but bigraded; and there are two common conventions for how to grade them.  We index them according to the following convention:
\begin{defin}
	Define \(S^{1,0}\) as the \(\mathbb{P}^1\)-suspension spectrum of the simplicial sphere \((S^1,1)\) and \(S^{1,1}\) as the \(\mathbb{P}^1\)-suspension spectrum of \((\mathbb{A}^1-0,1)\). The suspension spectrum of \(\mathbb{P}^1\) is then equivalent to \(S^{2,1}\). Define 
	\[S^{p,q}\defeq(S^{1,0})^{\wedge(p-q)}\wedge (S^{1,1})^q.\]
	This relates to the other common notation of \(S^{\alpha} = S^{1,1}\) by \(S^{p,q}=S^{p-q+q\alpha}\).\\
	The motivic homotopy groups of a motivic spectrum \(X\in\mathcal{SH}_k\) are then defined as:
	\[
	\pi_{p,q}(X)\defeq[S^{p,q},X]_{\mathcal{SH}_k}
	\]
\end{defin}
\ \\
There is a topological realization functor \(R: \mathcal{SH}_\C \rightarrow \mathcal{SH}_{top}\) called Betti realization. There are many reviews of the construction and basic properties of this functor. We rely on the account in \cite[4.3]{JOA}.
Betti realization maps the suspension spectrum of a smooth scheme over \C\ to the suspension spectrum of the topological space of its complex points, endowed with the analytic topology. In particular the image of the motivic sphere \(S^{p,q}\) under Betti realization is the topological sphere \(S^p\). This functor is a strict symmetric monoidal left Quillen functor. 

Because Betti realization maps the motivic spheres to the topological ones, it induces maps on homotopy groups
\[R: \pi_{pq}(X)\rightarrow \pi_p(R(X))\]
for every motivic spectrum \(X\in\mathcal{SH}_\C\). Therefore for every motivic spectrum \(E\in\mathcal{SH}_\C\) it also induces maps \[R: E_{pq}(X)\rightarrow R(E)_p(R(X))\]
and
\[R: E^{pq}(X)\rightarrow R(E)^p(R(X))\] on homology and cohomology  associated to that spectrum. Betti realization has a strict symmetric monoidal right inverse
\[
c: \mathcal{SH}_{top} \rightarrow \mathcal{SH}_\C
\]
called the constant simplicial presheaf functor. It is a result of Levine(see \cite[Theorem 1]{LEV}) that \(c\) is not only faithful but also full.
\subsection{Cellular motivic spectra}
We intent to construct an example \(v_n\)-self map \(v: \X_n\rightarrow \X_n\) on the motivic equivalent of the space \(X_n\) used by Hopkins and Smith. This space is constructed as a retract of a finite cell spectrum. In classical topology, a retract of a finite cell spectrum is a finite cell spectrum again, but this does not necessarily need to be the case motivically. Therefore, we want to consider the slightly larger thick envelope \(\mathcal{SH}_{\C}^{qfin}\) (defined in \ref{finiteType}) of the subcategory of finite spectra in \(\mathcal{SH}_\C\) in the definition of motivic \(v_n\)-self maps and for the study of thick subcategories characterized by \(v_n\)-self maps. 
In contrast to classical algebraic topology, not all motivic spectra are cellular in the following sense:
\begin{defin}
\label{finiteType}
\begin{enumerate}
\item The category of cellular spectra \(\mathcal{SH}_k^{cell}\) in \(\mathcal{SH}_k\) is defined (c.f. \cite[Definition 2.1]{DI2}) as the smallest full subcategory that satisfies
\begin{itemize}
\item The spheres \(S^{p,q}\) are contained in the subcategory \(\mathcal{SH}_k^{cell}\).
\item If a spectrum \(X\) is contained in the subcategory \(\mathcal{SH}_k^{cell}\), then so are all spectra which are weakly equivalent to \(X\).
\item If \(X\rightarrow Y \rightarrow Z\) is a cofiber sequence and two of the three spectra are contained in the subcategory \(\mathcal{SH}_k^{cell}\), then so is the third.
\item The subcategory \(\mathcal{SH}_k^{cell}\) is closed under arbitrary colimits.
\end{itemize}
\item The subcategory of \emph{finite cellular spectra} \(\mathcal{SH}_k^{fin}\) in \(\mathcal{SH}_k\) is defined similarly as the smallest full subcategory that satisfies the first three conditions (see \cite[Definition 8.1]{DI2}).\\
\item We define the category of \emph{quasifinite cellular spectra} \(\mathcal{SH}_k^{qfin}\) as the smallest full triangulated subcategory of \(\mathcal{SH}_k\) that contains \(\mathcal{SH}_k^{fin}\) and is closed under retracts. The spectra in \(\mathcal{SH}_k^{qfin}\) are exactly finite cell spectra and their retracts, since the cofiber of two retracts of finite cell spectra is a retract of a finite cell spectrum by the octahedral axiom.
\item By \cite[Lemma 2.2]{ROEN} a motivic spectrum is cellular if and only if it admits a cell presentation, i.e. it can be built by successively attaching cells \(S^{s,t}\). A motivic cell spectrum \(X\) is called \emph{of finite type} if it admits a cell presentation with the following property: there exists a \(k\in \N \) such that there are no cells in dimensions satisfying \(s-t<k\) and such that there exist only finitely many cells in dimensions \((s+t,t)\) for each \(s\).
\end{enumerate}
\end{defin}
\ \\

\subsection{Completions}
For the sake of studying periodic self maps it is useful to consider one prime at a time, because these maps are detected by a collection of cohomology theories called Morava K-theory, which are defined with regard to a specific prime. In our case this prime will usually be odd, i.e. different from two. Topologically one can implement this by studying the localized or completed homotopy category via the tool of Bousfield localization at an appropiate Moore spectrum. Motivically this works as well (a discussion of this in the motivic setting can be found in \cite[Section 3]{OR}): We define the \(l\)-completed motivic homotopy category \(\mathcal{SH}_{k,l}^\wedge\) as the Bousfield localization of the category \(\mathcal{SH}_k\) at the mod-l Moore spectrum \(S/l\).\\

\begin{defin}
Let \(l\) be any prime number, and let $X$ be a motivic spectrum in \(\mathcal{SH}_k\).  The \emph{$l$-completion \(X^{\wedge}_l\) of $X$}  is the Bousfield localization of $X$ at the mod-\(l\) Moore spectrum \(S/l\). One can also describe this completion as:
\[X^{\wedge}_l\defeq L_{S/l}X \simeq \underset{\leftarrow}{\text{holim}}~ X/{l^n}\]
\end{defin}

\begin{defin}
We define  the subcategory \(\mathcal{SH}_{k,l}^{\wedge,cell}\) of \emph{\(l\)-complete cellular spectra} in \(\mathcal{SH}_{k,l}^\wedge\) as the full subcategory of \(l\)-completions of cellular spectra. Similarly, we define the subcategories \(\mathcal{SH}_{k,l}^{\wedge,fin}\) of \(l\)-complete finite cellular spectra and \(\mathcal{SH}_{k,l}^{\wedge,qfin}\) of \(l\)-complete quasifinite cellular spectra as the full subcategories of \(l\)-completions of spectra in \(\mathcal{SH}_{k}^{fin}\) and \(\mathcal{SH}_{k}^{qfin}\).
\end{defin}

\subsection{Motivic Spanier-Whitehead duality}
We are going to make use of Spanier-Whitehead duality when we study periodic self maps. The sources we want to quote use different, but equivalent definitions of dualizability, so we collect a number of basic definitions and facts about Spanier-Whitehead duality that we are going to use in one place. Our primary source is \cite[III.1]{LMS} where categorical duality is explained with great detail.\\
Consider a spectrum \(X\) in \(\mathcal{SH}_k\) or \(\mathcal{SH}_{k,l}^{\wedge}\). Both categories are closed symmetric monoidal categories (see \cite{JAR}), and therefore for an arbitrary motivic spectrum \(Y\) there exists a function spectrum \(F(X,Y)\). The unit and counit of the canonical tensor-hom adjunction are given by maps
\[\eta_{X,Y}: X\rightarrow F(Y,X\wedge Y)\]
and by the evaluation
\[\epsilon_{X,Y}: F(X,Y)\wedge X\rightarrow Y\]
and furthermore there is a natural pairing
\[
F(X,Y)\wedge F(X',Y')\rightarrow F(X\wedge X',Y\wedge Y')
\]
which provides a natural map
\[\nu_{X,Y}:F(X,S)\wedge Y \rightarrow F(X,Y)\]
by specializing to the case \(X'=Y=S\) and using the fact \(F(S,Y')\cong Y'\).

\begin{prop}
\label{SWEQUI}
Let  \(X\) be a spectrum in \(\mathcal{SH}_k\) or \(\mathcal{SH}_{k,l}^{\wedge}\). Then the following three conditions are equivalent:
\begin{enumerate}
\item The canonical map \[\nu_{X,Y}:F(X,S)\wedge Y \rightarrow F(X,Y)\] is an isomorphism for all spectra \(Y\).
\item The canonical map \[\nu_{X,X}:F(X,S)\wedge X \rightarrow F(X,X)\] is an isomorphism.
\item There is a coevaluation map \(coev: S\rightarrow X \wedge F(X,S)\) such that the diagram
\[
\xymatrix@C=1.5cm@R=1.5cm
{
S \ar[r]^{coev}\ar[d]^{\eta_{S,X}}&
X \wedge F(X,S)\ar[d]^T\\
F(X,X)&
F(X,S)\wedge X\ar[l]_{\nu_{X,X}}\\
}
\]
commutes, where \(T\) denotes the transposition map.
\end{enumerate}
\end{prop}
\begin{proof}
Clearly the first point implies the second. The second point implies the third, because one can define \(coev\) as the composite \(T\circ\nu_{X,X}^{-1}\circ\eta_{S,X}\). Finally, the third point implies the first (c.f. \cite[Proposition III.1.3(ii)]{LMS}) because one can define an inverse to \[\nu_{X,Y}:F(X,S)\wedge Y \rightarrow F(X,Y)\]
as the following composite: 
\begin{align*}
\nu_{X,Y}^{-1}: F(X,Y)&\cong F(X,Y)\wedge S \overset{id\wedge coev}\longrightarrow F(X,Y)\wedge X \wedge F(X,S) \overset{\epsilon_{X,Y}\wedge id}\longrightarrow\\
&\longrightarrow Y\wedge F(X,S)\overset{T}\longrightarrow F(X,S)\wedge Y
\end{align*}
\end{proof}

\begin{defin}
If \(X\) satisfies any of the preceding conditions, it is called \emph{strongly dualizable}.\\
The spectrum \(DX=F(X,S)\) is called the \emph{(motivic) Spanier-Whitehead dual} of \(X\). By definition, \(D\defeq F(-,S)\) is a contravariant functor \[D: \mathcal{SH}_k\rightarrow \mathcal{SH}_k\] and similarly \(D\defeq F(-,S^\wedge_{l})\) is a contravariant functor  \[D:\mathcal{SH}_{k,l}^{\wedge}\rightarrow \mathcal{SH}_{k,l}^{\wedge}\]
on the category of \(l\)-complete spectra. In fact, the obvious map \[F(-,S)\rightarrow F(-,S^\wedge_{l})\]
is a completion at \(l\), but we will neither need nor prove it.
\end{defin}

We will need the following general facts about strongly dualizable spectra, which are proven in \cite[Proposition III.1.3 (i, iii)]{LMS}:
\begin{lemma}
\begin{enumerate}
	\item If \(X\) is strongly dualizable, then \(DDX\cong X\).
	\item If \(X\) and \(Y\) are strongly dualizable, then the natural map \[F(X,S) \wedge F(Y,S) \rightarrow F(X \wedge Y, S)\]
	is an isomorphism. In particular, \(X \wedge Y\) is strongly dualizable.
\end{enumerate}
\end{lemma}

The spectrum \(DX\wedge X\) has the structure of a homotopy ring spectrum by the same arguments as in\cite[Proof of Corollary 5.1.5]{RAV2}:
\begin{rem}
	\label{SWRingSpectrum}
	If \(X\) is strongly dualizable, then the unit map \[e: S \overset{\eta_{S,X}}{\longrightarrow} F(X,X) \cong F(X,S) \wedge X=DX\wedge X\]
	and the multiplication map 
	\[ \mu: DX \wedge X \wedge DX \wedge X \overset{D(e)}\longrightarrow DX\wedge S \wedge X \cong DX \wedge X\]
	endow \(DX\wedge X\) with the structure of motivic homotopy ring spectrum (in fact an \(A_\infty\)-structure, but we are not going to use or prove it), where we use
	\[X\wedge DX\cong DDX\wedge DX=D(DX\wedge X)\]
	in the definition of \(D(e)\).
\end{rem}

\begin{lemma}
\label{DCofib} The functor \(D\) maps cofiber sequences to cofiber sequences, and the full subcategory of strongly dualizable spectra in \(\mathcal{SH}_k\) is thick.
\end{lemma}
\begin{proof}
For the first statement, let \(X\rightarrow Y \rightarrow Z\) be a cofiber sequence. Because \(\mathcal{SH}_k\) is the homotopy category of a pointed monoidal model category, the functor \(F(-,A)\) maps cofiber sequences to fiber sequences for any \(A\) in \(\mathcal{SH}_k\) (c.f. \cite[6.6]{HOV}). In particular this is true for \(D(-)=F(-,S)\). Because \(\mathcal{SH}_k\) is stable, fiber and cofiber sequences agree, and \(DZ \rightarrow DY \rightarrow DX\) is a cofiber sequence again.\\
For the second statement we only need to show that a retract of a strongly dualizable spectrum is again strongly dualizable, so let \(A\) be a retract of a strongly dualizable spectrum \(X\). Note that by the first point of \ref{SWEQUI} we have to show that the canonical map
\[
F(A,S)\wedge Y \rightarrow F(A,Y)
\]
is an isomorphism for all motivic spectra \(Y\), and we already now this statement is true if we replace \(A\) with \(X\). But this follows immediately from the following diagram:
\[
\xymatrix{
F(X,S)\wedge Y \ar @`{(-10,10),(10,10)}^{id} \ar[r]^\cong\ar[d]&F(X,Y)\ar[d]\ar @`{(15,10),(35,10)}^{id}\\
F(A,S)\wedge Y \ar[r]\ar@/^/[u]& F(A,Y)\ar@/_/[u]\\
}
\]
\end{proof}

\begin{lemma}
	\label{CellDualizable}
All spectra in \(\mathcal{SH}_\C^{qfin}\) are strongly dualizable in \(\mathcal{SH}_\C\), and \(\mathcal{SH}_\C^{qfin}\) is closed under taking duals.\\
As a consequence, all spectra in \(\mathcal{SH}_{k,l}^{\wedge,qfin}\) are strongly dualizable in \(\mathcal{SH}_{\C(l)}\), and \(\mathcal{SH}_{k,l}^{\wedge,qfin}\) is closed under taking duals.
\end{lemma}
\begin{proof}
Finite cell spectra are contained in the thick subcategory of compact spectra, and compact spectra are dualizable(See \cite[Remark 4.1]{NOS} or \cite[5.2.7]{JOA}). Therefore the thick subcategory generated by finite cell spectra is dualizable.\\
To show that \(\mathcal{SH}_\C^{qfin}\) is closed under taking duals, we only have to check that the duals of finite cell spectra and their retracts are in \(\mathcal{SH}_\C^{qfin}\) again by \ref{finiteType}. This is true for finite cell spectra by cellular induction, because the duals of suspensions of the sphere spectrum are suspensions of the sphere spectrum.
If \(X\) is a retract of a spectrum \(F \in \mathcal{SH}_\C^{qfin}\) such that \(DF\in \mathcal{SH}_\C^{qfin}\), with maps \(r:F\rightarrow X\) and \(s:X\rightarrow F\) such that \(r\circ s=id_X\), then \(DX\) is a retract of \(DF\in \mathcal{SH}_\C^{qfin}\) with maps \(Ds:DF\rightarrow DX\) and \(Dr:DX\rightarrow DF\) because \(Ds\circ Dr=id_{DX}\).
\end{proof}

\subsection{The motivic Steenrod algebra and the dual motivic Steenrod algebra}
One key ingredient for the Adams spectral sequence is knowlegde of the Steenrod algebra or of the dual Steenrod algebra. Motivically, the Steenrod Algebra was described by Voevodsky for fields of characteristic zero and later by Hoyois, Kelly and {\O}stv{\ae}r in positive characteristic. While some interesting phenomenas happen at the prime two, the motivic Steenrod algebra is more closely related to the classical topological Steenrod algebra at odd primes. To describe the motivic Steenrod algebra it is sufficient to know the coefficients of motivic coholomogy with \(\Z/l\Z\)-coefficients:

\begin{prop}For \(l\neq 2\) a prime and \(k=\C\) the coefficients \(H\Z/l^{**}\) of motivic cohomology are given as a ring by \[H\Z/l^{**}\cong \Z/l [\tau] \] with \(|\tau|=(0,1)\), and the image of \(\tau\) under Betti realization is nonzero.
\end{prop}
\begin{proof}
We know that \(H\Z/l^{**}=0\) for $q<p$ ((cf. \cite[Theorem 3.6]{MVW})).\\ Let \(q\geq p\). Then there is an isomorphism from motivic to \'etale cohomology: \[H^{p,q}(Spec(k), \mathbb{Z}/l) \cong H_{\acute{e}t}^p(k,\mu_l^{\otimes q})\]
This isomorphism respects the product structure(\cite[1.2,4.7]{GL}).\\
The \'etale cohomology groups \(H_{\acute{e}t}^p(k,\mu_l^{\otimes q})\) can be computed as the Galois cohomology of the separable closure of the base field (in both cases the complex numbers) with coefficients in the $l$-th roots of unity. The action of the absolute Galois group $G$ is given by the trivial action if $k=\mathbb{C}$ and by complex conjugation if $k=\mathbb{R}$:
\[H_{\acute{e}t}^p(k,\mu_l^{\otimes q})\cong H(G, \mu_l^{\otimes q}(\C))\]
For $k=\mathbb{C}$, these groups all vanish for $p\neq 0$ by triviality of the Galois action, and they are $\mathbb{Z}/l$ in the degree $p=0$ for all $q\geq 0$. The multiplicative structure is given by the tensor product of the modules.\\
\end{proof}

\begin{rem}
We denote the image of \(\tau\) under \(H\Z/l_{**}=H\Z/l^{-*,-*}\) with the same name. This image has bidegree \(|\tau|=(0,-1)\).
\end{rem}

The motivic mod-$l$ Steenrod algebra over basefields of characteristic 0 has been computed by Voevodsky in \cite{VOE2}. The implications for the dual motivic Steenrod algebra are for example written down in the introduction of \cite{HKO}. In our special case it has the following shape:
\begin{prop}
Let $k=\C$ as above, and let $l$ be an odd prime. The dual motivic Steenrod algebra \(A_{**}\) and its Hopf algebroid structure can be described as follows:
\[A_{**}=H\Z/l_{**}[\tau_0,\tau_1,\tau_2,...,\xi_1,\xi_2,...]/(\tau_i^2=0)\]
Here \(|\tau_i|=(2l^i-1,l^i-1)\) and \(|\xi_i|=(2l^i-2,l^i-1)\).\\ The comultiplication is given by
\[\Delta(\xi_n) = \sum_{i=0}^n \xi_{n-i}^{l^i} \otimes \xi_i\] where \(\xi_0:=1\), and
\[\Delta(\tau_n) = \tau_n\otimes 1 + \sum_{i=0}^n \xi_{n-i}^{l^i} \otimes \tau_i\]
\end{prop}

\subsection{Generalized motivic Adams spectral sequences}
We will use the homological motivic Adams spectral sequence to compute the coefficients of the \(l\)-completed motivic Brown-Peterson spectrum \(ABP^\wedge_l\).
The motivic Adams spectral sequence was inspired by Morels computation of the zeroth motivic stable stem(c.f. \cite{Mor}) and was used by Dugger and Isaksen for extensive computations over \(\C\) at the prime 2 (c.f. \cite{DI}). They also use additional information available in the MASS to deduce new information about the classical Adams spectral sequence. In other work Isaksen has extended these computations to the base field \(\mathbb{R}\). Generalized motivic Adams spectral sequences can be constructed for \(E\) an arbitrary motivic ring spectrum and \(X\) a motivic spectrum.
Define \(\bar{E}\) as the cofiber of the unit map \(S\rightarrow E\). Smashing the cofiber sequence \(\bar{E} \rightarrow S \rightarrow E\)  with  \(\bar{E}^s\wedge X\) yields cofiber sequences
\[\bar{E}^{\wedge(s+1)}\wedge X \rightarrow \bar{E}^{\wedge s}\wedge X \rightarrow E\wedge \bar{E}^{\wedge s}\wedge X\]
giving  rise to the following tower, called the canonical \(E_{**}\)-Adams resolution:
\[
\xymatrix
{
...\ar[r]\ar[d]&
\bar{E}^{\wedge(s+1)} \wedge X \ar[r]\ar[d] &
\bar{E}^{\wedge s} \wedge X \ar[r]\ar[d]&
...\ar[r]\ar[d]&
\bar{E}\wedge X \ar[r]\ar[d]&
X\ar[d]\\
...&
E\wedge \bar{E}^{\wedge(s+1)} \wedge X &
E\wedge \bar{E}^{\wedge s} \wedge X &
...&
E\wedge \bar{E} \wedge X &
E \wedge X\\
}
\]
The long exact sequences of homotopy groups associated to these cofiber sequences forms a trigraded exact couple
\[
\xymatrix{
\pi_{**}(\bar{E}^{\wedge*}\wedge X) \ar[rr]&
&
\pi_{**}(\bar{E}^{\wedge*}\wedge X)\ar[dl]\\
&
\pi_{**}(E\wedge\bar{E}^{\wedge*}\wedge X)\ar[lu]
&\\
}
\]
and thus give rise to a trigraded spectral sequence \(E_r^{s,t,u}(E,X)\) with differentials \(d_r: E_r^{s,t,u} \longrightarrow E_r^{s+r,t+r-1,u}\).\\

If one furthermore assumes that $E_{**}E$ is flat as a (left) module over the coefficients \(E_{**}\) it is possible to identify the \(E_2\)-term via homological algebra. In this case one can associate a flat Hopf algebroid to \(E\) (See \cite[Lemma 5.1]{NOS} for the statement and \cite[Appendix 1]{RAV} for the definition and basic properties of Hopf algebroids), and the category of comodules over this Hopf Algebroid is abelian and thus permits homological algebra. Because \(E_{**}E\)  is flat over \(E_{**}\) there is also an isomorphism (see \cite[Lemma 5.1(i)]{NOS})
\[\pi_{**}(E\wedge E \wedge X)\cong E_{**}(E)\underset{E_{**}}{\otimes}E_{**}(X)\]
allowing us to identify the long exact sequences of homotopy groups of the canonical \(E_{**}\)-Adams resolution with the (reduced) cobar complex \(C^*(E_{**}(X))\). For this reason the resolution is also referred to as the geometric cobar complex. The \(E_2\)- page of the $E$-Adams spectral sequence can then be described as:
\[E_2^{s,t,u}(E,X)=\text{Cotor}^{s,t,u}_{E_{**}(E)}(E_{**},E_{**}(X))\]
Here \(\text{Cotor}\) denotes the derived functors of the cotensor product in the category of \(E_{**}(E)\)-comodules and can be computed as the homology of the cobar complex \(C^*(E_{**}(X))\).

\begin{rem}
Assume now that \(k=\C\). Then Betti realization induces a map of spectral sequences \(R_{E,X}: E_r^{s,t,u}(E,X)\rightarrow E_r^{s,t}(R(E),R(X))\)
\end{rem}
This can be checked by going through the definitions: Because Betti realization preserves cofiber sequences and smash products, we have \(R( \bar{E})=\overline{R(E)}\), and the realization of the canonical \(E_{**}\)-Adams resolution for \(X\) is the canonical \(R(E)_{*}\)-Adams resolution for the topological spectrum \(R(X)\). If we consider the induced maps on the long exact sequences of homotopy groups defining the exact couple, we get the following commutative diagram:
\[
\xymatrix@C-=0.3cm{
...\ar[r]\ar[d]^R&
\pi_{p,*}(\bar{E}^{\wedge(s+1)}\wedge X)\ar[r]\ar[d]^R&
\pi_{p,*}(\bar{E}^{\wedge s}\wedge X)\ar[r]\ar[d]^R&
\pi_{p,*}(E\wedge\bar{E}^{\wedge s}\wedge X)\ar[r]\ar[d]^R&
...\ar[d]^R\\
...\ar[r]&
\pi_{p}(\overline{R(E)}^{\wedge(s+1)}\wedge R(X))\ar[r]&
\pi_{p}(\overline{R(E)}^{\wedge s}\wedge R(X))\ar[r]&
\pi_{p}(R(E)\wedge \overline{R(E)}^{\wedge s}\wedge R(X))\ar[r]&
...
}\]
In particular, Betti realization induces a map of exact couples and hence a map of spectral sequences.\\

Convergence of the spectral sequence has been studied for the case \(E=H\Z/l\) by Hu, Kriz and Ormsby in \cite[Theorem 1]{HKO}. It turns out that over the complex numbers, the spectral sequence will just converge to the $l$-completion \(X^{\wedge}_l\) of $X$, which one can either describe  as the Bousfield localization of $X$ at the mod-\(l\) Moore spectrum \(S/l\) or explicitely as:
\[X^{\wedge}_l\defeq L_{S/l}X \simeq \underset{\leftarrow}{\text{holim}}~ X/{l^n}\]
The homotopy groups of $X$ and its $l$-completion are related by the following short exact sequence(\cite[End of section 3]{OR}):
\[0 \rightarrow \text{Ext}^1(\Z/{l^\infty},\pi_{**}X)\rightarrow \pi_{**}X^\wedge_l \rightarrow \text{Hom}(\Z/{l^\infty},\pi_{*-1,*}X) \rightarrow 0\]
We will see later that for our case of interest, where \(X=ABP^\wedge_l\), the spectral sequence actually converges strongly because of a vanishing line.

\subsection{The algebraic Morava-K-theories \(AK(n)\)}
As before we work over the complex numbers, and the prime \(l\) will be odd. In particular this prime is implicit in the definition of the motivic Brown-Peterson-spectrum \(ABP\) and of the algebraic Morava-K-theory spectrum \(AK(n)\). In this section we show that the algebraic Morava-K-theory spectra \(AK(n)\) admit the structure of a commutative homotopy ring spectrum similar to their classical counterparts. These spectra were originally defined by Borghesi in \cite{Bor}. In addition we rely on the description provided in \cite[Def. 6.3.1]{JOA}:
\begin{defin}
The connective n-th motivic Morava K-theory is defined as 
\[Ak(n)=ABP/(v_0,...,v_{n-1},v_{n+1},v_{n+2},...)\]
and the n-th motivic Morava K-theory spectrum \(AK(n)\) is defined as: \[AK(n)=v_n^{-1}ABP/(v_0,...,v_{n-1},v_{n+1},v_{n+2},...)\]
In particular, both spectra are \(MGL_{(l)}\)-modules.
\end{defin}
\(AK(n)\)  and \(Ak(n)\) are genuinely motivic in the sense that they are derived from the spectrum representing algebraic cobordism. We will need some of the properties of \(AK(n)\) proven in \cite{JOA}, namely:
\begin{rem}
\begin{enumerate}
\item The Betti realization of the (connective) motivic Morava K-theory is the classical (connective) Morava K-theory (\cite[Lemma 6.3.2]{JOA}): 
\[R_{\C}(AK(n))=K(n)\]
and
\[R_{\C}(Ak(n))=k(n)\]
\item By \cite[Lemma 6.3.7]{JOA} the coefficients of algebraic Morava K-theory are given by: \[AK(n)_{**}=H\Z/(l)_{**}\underset{\Z/(l)}{\otimes}K_{*}\]
\item If \(X\) is a finite motivic cell spectrum such that \(H\Z/(l)^{**}(X)\) is free over the coefficients, then the motivic Adams spectral sequence for \(Y=Ak(n)\wedge X\) will converge strongly to \(Ak(n)_{**}(X)\). (See \cite[8.3.3]{JOA})
\end{enumerate}
\end{rem}

At least for odd primes, the topological spectra \(K(n)\) can be shown to be homotopy ring spectra. As remarked in \cite[Remark 6.3.3(6)]{JOA}, it is not known in general if the motivic Morava K-theory spectrum \(AK(n)\) can be endowed with the structure of a motivic homotopy ring spectrum. In the special case \(k=\C\), \(l\neq 2\) however Joachimi proved that the spectrum
\[AP(n)\defeq ABP/(v_0=l,v_1,...,v_{n-1}),\]
another quotient of \(MGL\), admits a unital homotopy associative product \cite[9.3]{JOA}, and with the work done by her it is no longer difficult to do the same for \(AK(n)\).\\

We want to use and extend the results in \cite[9.3]{JOA} and follow the notation used there to make comparison easier. In particular \(\eta\) will not denote the motivic Hopf map in this chapter, but a different map to be defined later. The only exception is the name of the prime \(l\), which is referred to as \(p\) in \cite{JOA}.\\
Let \(R \in \mathcal{SH}_k\) be a strictly commutative ring spectrum with multiplication map \(m: R\wedge R \rightarrow R\) and unit map \(i:S\rightarrow R\). The example that we have in mind is \(MGL_{(l)}\), which is a strictly commutative motivic ring spectrum by the reasoning given in the beginning of \cite[9.3]{JOA}.\\

Classically one can study the products on \(R\)-modules of the form \(R/x\) and use them to gain information about products on quotients of the form \(R/X\) where \(X\) is a countable regular sequence of homogeneous elements.
In contrast to the classical situation, the coefficients of \(MGL_{(l)}\) are not known, but the coefficients \(MGL_{(l)}/l\) are. Therefore motivically one has to consider \(R\)-modules of the form \(R/(x,y)\).\\

In the section immediately preceding \cite[9.3.7]{JOA} and in the proof of \cite[Lemma 9.3.8]{JOA} Joachimi constructs a product on quotients of this form and proves the following statement:
\begin{lemma}
\label{ProductLemma}
Let \(y\in\pi_{k',l'}(R)\) and let \(x\in \pi_{k,l}(R)\). Define the \(R\)-modules \(M\defeq R/y\) and \(N\defeq M/x\) and denote the structure map of \(M\) as \(\nu_M: R\wedge M\rightarrow M\). Write \(\eta'\) for the canonical map \[\eta': R\rightarrow M=R/y\] and \(\eta\) for the canonical map \[\eta: M\rightarrow N=R/(x,y)\].\\

If \(\pi_{2k'+1,2l'}(M)=0\) and \(\pi_{2k+1,2l}(N)=0\), there are maps of \(R\)-modules
\[\mu_M: M\wedge M\rightarrow M\]
\[\nu_{M,N}: M\wedge N \rightarrow M \]
\[\mu_N: N\wedge N \rightarrow N\]
making the following diagrams commute up to homotopy:
\begin{equation}
\label{ProductLemmaDiagram1}
\xymatrix{
R \wedge R \ar[rr]^{\eta'\wedge \eta'} \ar[dd]^{m} \ar[dr]^{1\wedge \eta'} & & M\wedge M \ar[dd]^{\mu_M}\\
 & R \wedge M \ar[dr]^{\nu_M} \ar[ur]^{\eta'\wedge 1}& \\
R \ar[rr]^{\eta'} & & M\\
}
\end{equation}

\begin{equation}
\label{ProductLemmaDiagram2}
\xymatrix{
	M \wedge M \ar[rr]^{\eta\wedge \eta} \ar[dd]^{\mu_M} \ar[dr]^{1\wedge \eta} & & N \wedge N \ar[dd]^{\mu_N} \\
	& M \wedge N \ar[dr]^{\nu_{M,N}} \ar[ur]^{\eta\wedge 1}& \\
	M \ar[rr]^{\eta} & & N\\
}
\end{equation}
In particular, if we choose the maps \(\eta'\circ i\) and \(\eta\circ\eta'\circ i\) as unit maps, \(\mu_M\) and \(\mu_N\) are unital products on \(M\) and \(N\) respectively.
\end{lemma}

Furthermore the following result of Joachimi \cite[Lemma 9.3.8]{JOA} proves associativity, and we wish to extend it to include commutativity:
\begin{lemma}
	\label{HomotopyAssociativityLemma}
	If \(\pi_{k'+1,l'}(M)=\pi_{2k'+2,2l'}(M)=\pi_{3k'+3,3l'}(M)=0\), then \(\mu_M\) is homotopy associative.\\
	If furthermore \(\pi_{k+1,l}(N)=\pi_{2k+2,2l}(N)=\pi_{3k+3,3l}(N)=0\), then \(\mu_N\) is also homotopy associative.
\end{lemma}

We need the following lemma of Joachimi \cite[Lemma 9.3.3]{JOA} in the proof of commutativity:
\begin{lemma}
	\label{Joachimi933}
	Let \(R'\) be a (homotopy) ring spectrum, \(M'\) a left \(R'\)-module, and \(\pi_{k,l}(M')=0\). Then any \(R'\)-module map \(\psi:S^{k,l}\wedge R' \rightarrow M'\) is homotopically trivial.
\end{lemma}

\begin{prop}
\label{CommutativityLemma}
Let R be a homotopy ring spectrum and commutative up to homotopy. Let M and N be quotient modules defined as in \ref{ProductLemma}.

If \(\pi_{k'+1,l'}(M)=\pi_{2k'+2,2l'}(M)=0\), then \(\mu_M\) is homotopy commutative.\\
If furthermore the homotopy groups of \(N\) satisfy  \(\pi_{k+1,l}(N)=\pi_{2k+2,2l}(N)=0\), then \(\mu_N\) is also homotopy commutative.
\end{prop}
\begin{proof}
The \(R\) module \(M=R/y\) is defined by the following cofiber sequence:
\[
\Sigma^{k',l'}R\overset{\phi}\rightarrow R \overset{\eta'}\rightarrow M \overset{\delta}{\rightarrow} \Sigma^{k'+1,l'}R
\]

Recall that \(m:R\wedge R\rightarrow R\) is the product on the ring spectrum \(R\). To show that the product \(\mu_M: M\wedge M\rightarrow M\) is commutative, it suffices to show \[\theta\defeq\mu_M\circ(1-T):M\wedge M\rightarrow M\]
is homotopic to the zero map, where \(T\) is the transposition map. The map \[\theta'\defeq (\eta' \wedge id_M)\circ \theta\]fits into the following diagram of R modules
\[
\xymatrix{
	R\wedge R \ar[dd]_{id_R\wedge \eta'}\ar[rr]^{m \circ (1-T) = 0} \ar[ddrr]^{0} && R \ar[dd]^{\eta'}\\\\
	R\wedge M \ar[dd]\ar[rr]^{\theta'}&&M\\\\
	\Sigma^{k'+1,l'}R\wedge R \ar@{-->}[uurr]^{\bar{\theta'}}
}
\]
which commutes by \ref{ProductLemmaDiagram1}. The top horizontal map is zero up to homotopy because \(m\) is homotopy commutative by assumption, and the first column is the cofiber sequence defining \(M\), smashed with \(R\). Together, this implies the existence of the dashed map \(\bar{\theta'}\).

Now \(R\) is a \(R \wedge R\) module via the product map \(m\) and we can consider this diagram as a diagram of \(R \wedge R\) modules. Then proposition \ref{Joachimi933}, applied to the ring spectrum \(R \wedge R\), implies that \(\bar{\theta'}\) is null homotopic by our assumptions on the homotopy groups of \(M\). Therefore \(\theta'\) is null homotopic as well. We then get the following commutative diagram for \(\theta\): \[
\xymatrix{
	&&R\wedge M \ar[dd]_{\eta'\wedge id_M}\ar[rr]^{\theta' = 0} \ar[ddrr]^{0} && M \ar@{=}[dd]\\\\
	&&M\wedge M \ar[dd]\ar[rr]^{\theta}&&M\\\\
	\Sigma^{k'+1,l'}R\wedge R \ar[rr]^{id_R \wedge \eta'} &&\Sigma^{k'+1,l'}R\wedge M \ar@{-->}[uurr]^{\bar{\theta}}\ar[rr] && \Sigma^{2k'+2,2l'}R\wedge R\ar@{.>}^{\tilde{\theta}}[uu]
}
\]
Once again the first column is a cofiber sequence, which implies the existence of the dashed map.
The composite \(\bar{\theta}\circ(id_R \wedge \eta')\) is null homotopic because this diagram is a diagram of \(R\wedge R\) modules again, so we can use the same argument as before. This implies the existence of the dotted map \(\tilde{\theta}\). This map also vanishes by the second condition on the homotopy groups of M, which in turn implies that \(\bar{\theta}\) is zero up to homotopy. Therefore \(\theta\) also vanishes, so \(\mu_M\) is homotopy commutative.

Because we used only the fact that \(R\) is a homotopy ring spectrum and not strict commutativity, and because diagram \ref{ProductLemmaDiagram2} in \ref{ProductLemma} commutes, we can then repeat the same proof with \(M\) replacing \(R\) and \(N\) replacing \(M\). Note that this would not have been possible if we worked over \(R\)-modules, because it is not clear that \(R/x\) is a strictly commutative ring spectrum again.
\end{proof}

\begin{lemma}
\label{CommutativityLemma2}
Let \(k=\C\) and \(l\neq 2\). The spectrum \(AP(n)\) admits a unital, homotopy associative and homotopy commutative product
\[\mu_{AP(n)}:AP(n)\wedge AP(n) \rightarrow AP(n)\]
and so do the spectra \(A_i=AP(n)/(v_{n+1},...,v_{n+i})\).
\end{lemma}
\begin{proof}
Except for the statement about commutativity, the first part of this lemma is the content of \cite[9.3.9]{JOA}. The essential argument in the proof of the cited lemma is as follows: if one has a sequence of elements \(J\subset R_{**}\) and one knows that \(A\defeq R/(J-\{x,y\})\) is a homotopy associative and commutative ring spectrum, then one can describe the product on \(R/J\cong A/(x,y) \cong A \underset{R}\wedge R/(x,y)\) by
\[(N\underset{R}\wedge A) \wedge (N\underset{R}\wedge A) \overset{\tau}\longrightarrow (N\wedge N)\underset{R}\wedge (A\wedge A)\overset{id_N\wedge id_N \wedge \mu_A}\longrightarrow N\wedge N\underset{R}\wedge A \overset{\mu_N \wedge id_A}\longrightarrow N\wedge A\]
and thus has to prove the vanishing of the obstruction groups to associativity only after application of \((-)\underset{R}\wedge A\) to the associativity diagram.\\
Now choose \(R=MGL_{(l)}\) and \(A=ABP\) and \(J\) such that \(MGL_{(l)}/J=AP(n)\). Then the relevant obstruction groups are trivial because for odd primes \(l\neq 2\), \(ABP_{**}\) is concentrated in bidegrees where the first degree is divisible by 4. We can then show that there is a homotopy associative product on \(AP(n)\) by induction; because \(AP(n)/(v_0,...,v_n)\), we only have to do finitely many steps,  and we can use the fact (see \cite[Lemma 9.3.7]{JOA}) that for any sequence \((l)\subset J'\):
\[
ABP/(J'\cup \{y\})\cong MGL_{(l)}/(l,y)\underset{MGL_{(l)}}\wedge ABP/J'\]
We can use the same argument to prove commutativity: if we apply \((-)\underset{R}\wedge A\) to all the relevant diagrams in \ref{CommutativityLemma}, we see that the obstructions to commutativity lie in groups \(\pi_{i,j}(M\underset{R}\wedge A)\) and \(\pi_{i,j}(N\underset{R}\wedge A)\) which are trivial because 4 does not divide \(i\) in the relevant bidegrees. Therefore the product on \(AP(n)\) is in fact homotopy commutative.\\
Now consider the spectra \(A_i\). To define them, we add finitely many elements, namely \(v_{n+1},...,v_{n+i}\), to the sequence \(J\). The proof of \cite[Lemma 9.3.7]{JOA} carries through verbatim  and we can conclude that there is a product map \(A_i\wedge A_i \rightarrow A_i\). Similarly, because we had to add only finitely many elements to \(J\), we can repeat the induction argument above for the spectra \(A_i\). This shows that the multiplication on \(A_i\) is in fact homotopy associative and homotopy commutative.
\end{proof}

By essentially classical arguments, this allows us to conclude that \(Ak(n)\) has the desired ring structure:
\begin{prop}
	Let \(k=\C\) and let \(l\) be an odd prime. Then the connective algebraic Morava K-theory spectrum
	\[Ak(n)=\underset\longrightarrow{\textnormal{hocolim }} A_i=ABP/(v_0,v_1,...v_{n-1},v_{n+1},v_{n+2},...)\]
	admits the structure of a homotopy associative and homotopy commutative motivic ring spectrum.
\end{prop}
\begin{proof}
By \cite[Corollary 9.3.5]{JOA} the elements \(v_i, i\neq n\) act trivially on \(Ak(n)\). This is in particular the case for \(v_0=l\). Therefore \cite[Lemma 6.7]{STR} holds for \(A=M=Ak(n)\) (although Strickland considers rings in R-modules, the only necessary modification is replacing the map \(\rho^*\) by the map \(\rho^*: [R/(l,x_i)\underset{R}\wedge B, M]\rightarrow[R/x_i\underset{R}\wedge B, M]\rightarrow[B,M]\)), and we can use the arguments of \cite[Proposition 6.8]{STR} to conclude that the constructed products on \(A_i\) induce a unital, homotopy associative product on \(Ak(n)\). As noted in the proof of Stricklands proposition, this product is commutative if and only if the maps \(A_i\rightarrow Ak(n)\) commute with themself (see \cite[Definition 6.1]{STR} for a definition of this notion). Because the product on \(A_i\) is commutative, this is the case for every map out of \(A_i\).
\end{proof}

\begin{corol}
Let \(k=\C\) and let \(l\) be an odd prime. The algebraic Morava K-theory spectrum \(AK(n)=v_n^{-1}Ak(n)\) admits the structure of a commutative and associative motivic homotopy ring spectrum.
\end{corol}
\begin{proof}
We have an isomorphism \(AK(n)\cong v_n^{-1}MGL_{(l)}\underset{MGL_{(l)}}\wedge Ak(n)\) and both smash factors admit a homotopy commutative and associative product(\cite[Proposition 6.6]{STR}). Therefore we can endow \(AK(n)\) with the desired structure as in the proof of \ref{CommutativityLemma2}.
\end{proof}

It remains to show that this product induces the same product structure on \(AK(n)_{**}\) as one would expect from the computation of these coefficients:
\begin{lemma}
The multiplication map \[\mu_{AK(n)}: AK(n)\wedge AK(n)\rightarrow AK(n)\]
induces the multiplication on \(AK(n)_{**}\) given by the multiplication on \(K(n)_*\) and the isomorphism \(AK(n)_{**}\cong HZ/l_{**}\underset{\Z/l}\otimes K(n)_*\) of \cite[Lemma 6.3.7]{JOA}.
\end{lemma}
\begin{proof}
The proof is similar to the proof of \cite[Lemma 9.3.10]{JOA}
\end{proof}

\section{Thick subcategories characterized by motivic \(v_n\)-self maps}
Let \(l\) be an odd prime and let \(k=\C\). The aim of this section is to show that the existence of \(v_n\)-self maps characterizes thick subcategories in \(\mathcal{SH}_{\C}^{qfin}\) and hence also in the motivic homotopy category. We consider only the case \(n>0\).\\

\begin{defin}
Let \(X\) be a motivic spectrum in \(\mathcal{SH}_{\C}^{qfin}\) or \(\mathcal{SH}_{\C,l}^{\wedge,qfin}\). A map \(f: \Sigma^{t,u} X \rightarrow X\) is a motivic \(v_n\) self-map if it satisfies the following conditions:
\begin{enumerate}
\item \(AK(m)_{**} f\) is nilpotent if \(m \neq n\)
\item \(AK(m)_{**} f\) is given by multiplication with an invertible element of \(H\Q_{**}\) if \(m=n=0\).
\item \(AK(m)_{**} f\) is an isomorphism if \(m=n\neq 0\).
\end{enumerate}
\end{defin}

\ \\
As mentioned before, the topological nilpotence theorem is a key ingredient in the proof that topological finite cell spectra spectra admitting a \(v_n\)-self map form a thick subcategory: A map of finite spectra is nilpotent if and only if it induces zero in all Morava K-theories. The motivic equivalent of this theorem does not hold: For example, the motivic Hopf map \(\eta\) is a non-nilpotent map in \(\mathcal{SH}_{\C}\), but induces the zero map in motivic Morava K-theory for degree reasons. It seems likely however that a weaker version of the theorem applies, where we only consider maps of a certain bidegree. For the remainder of this subsection we assume that the following motivic nilpotence conjecture holds:
\begin{conj}
	\label{NilpotenceConjecture}
Let \(k=\C\), let \(l\) be an odd prime and \(n>0\) be an integer. If \(X\) is a motivic spectrum in \(\mathcal{SH}_{\C}^{qfin}\) or \(\mathcal{SH}_{\C,l}^{\wedge,qfin}\) and \(f:\Sigma^{p,q}X\rightarrow X\) is a motivic map such that \((p,q)\) is a multiple of \((2l^n-2,l^n-1)\), then: \[\forall m\in \N: AK(m)_{**}(f)=0 \implies \exists k\in \N: f^k\simeq 0\]
\end{conj}
\ \\
The known examples of non-nilpotent motivic self maps that induce the zero map in motivic Morava-K-theory (A variety of examples can be found in \cite{HOR}, and Boghdan George has constructed a whole family of such maps detected by exotic motivic Morava K-theories) do not contradict this conjecture.\\
\ \\
To prove the motivic equivalent of asymptotic uniqueness, we want to use Betti realization to compare the motivic situation to the classical one. To do this, we need to study the effect of Betti Realization on homology groups of \(AK(n)\). We will show that the kernel of the map induced by Betti realization is precisely the \(\tau\)-primary torsion elements. To do this, we need to compare \(K(n)_*\) and \(AK(n)_{**}\)-modules, which is only possible after inverting \(\tau\). We also need the fact that the \(AK(n)\)-homology of a (quasi)-finite motivic cell spectrum is finitely generated over the coefficients:

\begin{lemma}
\label{FinGen}
Let \(X\) be a motivic spectrum in \(\mathcal{SH}_{\C}^{qfin}\) or \(\mathcal{SH}_{\C,l}^{\wedge,qfin}\). Then 
\begin{enumerate}
\item \(AK(n)_{**}(X)\) is finitely presented as an \(AK(n)_{**}\)-module.
\item \(\text{Hom}_{AK(n)_{**}}(AK(n)_{**}(X),M)\) is finitely presented as an \(AK(n)_{**}\)-module for every finitely presented \(AK(n)_{**}\)-module \(M\). In particular, \[\text{End}_{AK(n)_{**}}(AK(n)_{**}(X))\]
is finitely presented.
\end{enumerate}
\end{lemma}
\begin{proof}
Note that \(AK(n)_{**}\) is a quotient of a polynomial ring in the three variables \(v_n, v_n^{-1}, \tau\) over the field \(\mathbb{F}_l\) and hence Noetherian. Therefore a \(AK_{**}\)-module is finitely presented if and only if it is finitely generated. 
\begin{enumerate}
\item We will show the statement for finite cell spectra by cellular induction and then show that it also holds for retracts of finite cell spectra.
The claim is trivially true for the sphere spectrum. If the statement holds for a spectrum, it also holds for retracts of this spectrum because \(AK(n)_{**}\) is Noetherian and submodules of finitely generated modules are again finitely generated.
It remains to show that if the spectra \(X\) and \(Y\) in a cofiber sequence \(X\overset{f}\longrightarrow Y \overset{g}\longrightarrow Z \) satisfy the statement, then so does \(Z\). Consider the long exact sequence in \(AK(n)\)-homology
\[...\rightarrow  AK(n)_{**}(Y) \overset{g}\longrightarrow AK(n)_{**}(Z)\overset{\delta}\longrightarrow AK(n)_{*-1,*}(X)\overset{f}\longrightarrow AK(n)_{*-1,*}(Y)\rightarrow ...\]
associated to this cofiber sequence. We can break it up into short exact sequences in the canonical way:
\[0\rightarrow \textnormal{coker}(f)\overset{\bar{g}}\longrightarrow AK(n)_{**}(Z)\overset{\bar{\delta}_p}\longrightarrow \textnormal{ker}(f)[-1]\rightarrow 0\]
The two outer terms in the short exact sequence are finitely generated: \(\textnormal{ker}(f)[-1]\) as a submodule of a finitely generated module over a Noetherian ring, and \(\textnormal{coker}(f)\) as a quotient of a finitely generated module. Therefore the middle term is also finitely generated.
\item By the first part of this lemma, \(AK(n)_{**}(X)\) is finitely generated as an \(AK(n)_{**}\)-module. Therefore there is a surjection \(R^k\rightarrow AK(n)_{**}\) from a free and finitely generated \(AK(n)_{**}\)-module \(R^k\) onto \(AK(n)_{**}(X)\). Then
\[\text{Hom}_{AK(n)_{**}}(R^k,M)\cong M^k\]
is a free and finitely generated \(AK(n)_{**}\)-module. Because \(AK(n)_{**}\) is a Noetherian ring,
\[\text{Hom}_{AK(n)_{**}}(AK(n)_{**}(X),M)\]
is finitely generated as a submodule of this finitely generated module.
\end{enumerate}
\end{proof}

\begin{rem}
\label{ExtScalarFlat}
One can regard \(K(n)_{*}\) and its modules as a bigraded ring and bigraded modules concentrated in degree 0 with respect to the second bidegree. Then every \(AK(n)_{**}[\tau^{-1}]\)-module has the structure of a bigraded \(K(n)_{*}\)-module where \(v_n^{top}\) acts via \(\tau^{l^n-1}v_n\). (This of course implies that \((v_n^{top})^{-1}\) acts via \(\tau^{-l^n+1}v^{-1}_n\), so it only makes sense after inverting \(\tau\).) With this module structure, \(AK(n)_{**}[\tau^{-1}]\) is free (with basis \(\tau^{k}, k\in \mathbb{Z},- l^n+1< k < l^n-1\)) and in particular flat as a \(K(n)_{*}\)-module. Likewise it is flat as an \(AK(n)_{**}\)-module, because it is a localization. We will implicitly use this in the following statements and sometimes write \(-[\tau^{-1}]\) for \(-\underset{AK(n)_{**}}{\otimes}AK(n)_{**}[\tau^{-1}]\), and \(-[\tau,\tau^{-1}]\) for \(-\underset{K(n)_*}{\otimes}AK(n)_{**}[\tau^{-1}]\) as an abbreviation. 
\end{rem}

\begin{lemma}
(Compare \cite[2.7 + 2.8]{DI})\\
\label{taurealization}
\begin{enumerate}
\item Let \(X\) be a motivic spectrum in \(\mathcal{SH}_{\C}^{qfin}\) or \(\mathcal{SH}_{\C,l}^{\wedge,qfin}\).\\
We can define a map of bigraded \(AK(n)_{**}[\tau^{-1}]\)-modules (even a map of bigraded algebras if \(X\) is a ring spectrum) natural in \(X\)
\[R: AK(n)_{**}(X)\underset{AK(n)_{**}}{\otimes}AK(n)_{**}[\tau^{-1}]\rightarrow K(n)_{*}(R_\C(X))\underset{K(n)_*}{\otimes}AK(n)_{**}[\tau^{-1}]\]
via the assignment
\[x\otimes \tau^k\mapsto R_\C(x)\otimes \tau^{-q+k} \]
where \(q\) is the motivic weight of \(x\in AK(n)_{p,q}(X)\).\\
This map is an isomorphism.
\item The induced map
\[\bar{R}_{End,X}: \text{End}_{AK(n)_{**}}(AK(n)_{**}(X))[\tau^{-1}]\rightarrow \text{End}_{K(n)_{*}}(K(n)_{*}(R_\C(X)))[\tau,\tau^{-1}]\]
is an isomorphism of bigraded \(AK(n)_{**}[\tau^{-1}]\)-algebras.
\item  A homogeneous element \(f\in\text{End}_{AK(n)_{**}}(AK(n)_{**}(X))\) maps to zero under the map
\[
R_{End,X}: \text{End}_{AK(n)_{**}}(AK(n)_{**}(X))\rightarrow \text{End}_{K(n)_{*}}(K(n)_{*}(R_\C(X)))
\]
induced by motivic realization if and only if it is \(\tau\)-primary torsion.
\end{enumerate}
\end{lemma}
\begin{proof}
\begin{enumerate}
\item The statement about naturality and the module/algebra structure follow from the properties of motivic realization. It remains to show that the map is an isomorphism for spectra \(X\) in \(\mathcal{SH}_{\C}^{qfin}\) or \(\mathcal{SH}_{\C,l}^{\wedge,qfin}\). We will prove this using cellular induction, and then show that it remains an isomorphism under taking retracts.\\
Consider the case of the sphere spectrum \(X=S\): The map \[R: AK(n)_{**}[\tau^{-1}]\rightarrow K(n)_{*}\underset{K(n)_*}{\otimes}AK(n)_{**}[\tau^{-1}]\] sends \(\tau\) to \(\tau\) and \(v_n\in AK(n)_{2(l^n-1),l(n-1)}\) to \(v_n^{top}\otimes \tau^{-l^n+1}=1\otimes \tau^{-l^n+1}\tau^{l^n-1}v_n=v_n\), so it is  an isomorphism.\\
If \(X\) is a retract of a spectrum \(F\) for which the statement holds, then \(AK(n)_{**}(X)\) is a direct summand of \(AK(n)_{**}(F)\) and all squares in the following diagram commute:
\[
\xymatrix{
AK(n)_{p,q}(X)[\tau^{-1}]\ar[r]^{AK(n)_{**}(s)}\ar[d]^{R_X} \ar@/^2.0pc/[rr]_{id}& 
AK(n)_{p,q}(F)[\tau^{-1}]\ar[d]_{R_F}^\cong \ar[r]^{AK(n)_{**}(r)} &
AK(n)_{p,q}(X)[\tau^{-1}]\ar[d]^{R_X}
\\
K(n)_{p}(R_\C(X))[\tau,\tau^{-1}]\ar[r]_{K(n)_{*}(R_\C(s))}\ar@/_2.0pc/[rr]_{id} &
K(n)_{p}(R_\C(F))[\tau,\tau^{-1}]\ar[r]_{K(n)_{*}(R_\C(r))}&
K(n)_{p}(R_\C(X))[\tau,\tau^{-1}]
\\
}
\]

Therefore \(R_X\) is surjective and injective via a simple diagram chase.\\
Finally, suppose \(X\rightarrow Y \rightarrow Z\) is a cofiber sequence and the statement holds for \(X\) and \(Y\). Then the long exact sequence for \(AK(n)\)-homology maps to the long exact sequence for \(K(n)\)-homology associated to the cofiber sequence \(R_\C(X)\rightarrow R_\C(Y) \rightarrow R_\C(Z)\), and the five lemma tells us that the statement also holds for \(Z\):
\[
\xymatrix@C=7pt{
... \ar[r]\ar[d]_{\cong}^{R_X}&
AK(n)_{pq}(Y)[\tau^{-1}] \ar[r]\ar[d]_{\cong}^{R_Y}&
AK(n)_{pq}(Z)[\tau^{-1}] \ar[r]\ar[d]^{R_Z}&
AK(n)_{p-1,q}(X)[\tau^{-1}]\ar[r] \ar[d]_{\cong}^{R_X}&
... \ar[d]_{\cong}^{R_Y}\\
... \ar[r]&
K(n)_{p}(R_\C(Y))[\tau][\tau^{-1}] \ar[r]&
K(n)_{p}(R_\C(Z))[\tau][\tau^{-1}] \ar[r]&
K(n)_{p-1}(R_\C(X))[\tau][\tau^{-1}]\ar[r]&
...\\
}
\]

\item Let \(M\) be a finitely presented \(K(n)_{*}\)-module and \(N\) be an arbitrary \(K(n)_{*}\)-module. As noted in \ref{ExtScalarFlat}, \(AK(n)_{**}[\tau^ -1]\) is a flat \(K(n)_{*}\)-module. By \cite[§2.10, Proposition 11]{BOUR} there is a canonical isomorphism:
\[\text{Hom}_{K(n)_{*}}(M,N)[\tau,\tau^{-1}]\overset\cong\longrightarrow \text{Hom}_{K(n)_{*}[\tau,\tau^{-1}]}(M[\tau,\tau^{-1}],N[\tau,\tau^{-1}])\]
Likewise, let \(M\) be a finitely presented \(AK(n)_{**}\)-module and \(N\) be an arbitrary \(AK(n)_{**}\)-module. Because \(AK(n)_{**}[\tau^{-1}]\) is a flat \(AK(n)_{**}\)-module, there is also a canonical isomorphism:
\[\text{Hom}_{AK(n)_{**}}(M,N)[\tau^{-1}]\overset\cong\longrightarrow \text{Hom}_{AK(n)_{**}[\tau^{-1}]}(M[\tau^{-1}],N[\tau^{-1}])\] 

The module \(AK(n)_{**}(X)\) is finitely presented by \ref{FinGen}. Specializing to the case \(M=N=AK(n)_{**}(X)\), these two isomorphisms fit in the following commutative diagram:
\[
\xymatrix{
\text{End}_{AK(n)_{**}}(AK(n)_{**}(X))[\tau^{-1}] \ar[r]^\cong\ar[d]&
\text{End}_{AK(n)_{**}[\tau^{-1}]}(AK(n)_{**}(X)[\tau^{-1}]) \ar[d] \\
\text{End}_{K(n)_{*}}(K(n)_{*}(R_\C(X)))[\tau,\tau^{-1}] \ar[r]^\cong &
\text{End}_{K(n)_{*}[\tau,\tau^{-1}]}(K(n)_{*}(R_\C(X)))[\tau,\tau^{-1}])\\
}
\]
The first statement of the lemma tells us that \(K(n)_{*}[\tau,\tau^{-1}]\cong AK(n)_{**}[\tau^{-1}]\) and \(K(n)_{*}(R_\C(X)))[\tau,\tau^{-1}]\cong AK(n)_{**}(X)[\tau^{-1}]\), so the right vertical map is an isomorphism. It follows that the left vertical map is also an isomorphism.

\item Let \[P: \text{End}_{K(n)_{*}}(K(n)_{*}(R_\C(X)))[\tau,\tau^{-1}]\rightarrow \text{End}_{K(n)_{*}}(K(n)_{*}(R_\C(X)))\] be the map of \(K(n)_{**}\)-algebras defined by sending \(\tau\) to 1 and elements of \(\text{End}_{K(n)_{*}}(K(n)_{*}(R_\C(X)))\) to themselves. Then we have a commutative diagram of \(K(n)_{*}\)-algebras:
\[
\xymatrix{
\text{End}_{AK(n)_{**}}(AK(n)_{**}(X))[\tau^{-1}]\ar[d]^{\bar{R}_{End,X}}& 
\text{End}_{AK(n)_{**}}(AK(n)_{**}(X))\ar[l]\ar[d]^{R_{End,X}}\\
\text{End}_{K(n)_{*}}(K(n)_{*}(R_\C(X)))[\tau,\tau^{-1}]\ar[r]^{\quad P}& \text{End}_{K(n)_{*}}(K(n)_{*}(R_\C(X)))\\
}
\]

A homogeneous element maps to zero under the top horizontal map if and only if it is \(\tau\)-primary torsion. By the second statement of this lemma, the left vertical map is an isomorphism, and there are no homogeneous elements in the kernel of \(P\). All this together implies the desired result.
\end{enumerate}
\end{proof}

\begin{rem}
If \(X\) is strongly dualizable, the map \(DX\wedge X= F(X,S)\wedge X \rightarrow F(X,X)\) is a weak equivalence, and we have a corresponding isomorphism on homotopy groups \(\pi_{pq}(X\wedge DX)\cong \textnormal{End}(X)_{pq}\). With regard to motivic Morava K-theory the situation is more complicated. Using Spanier-Whitehead duality we have:
\begin{align*}AK(n)_{pq}(X\wedge DX)&=[S, AK(n) \wedge X \wedge DX]_{pq}\\
&=[X, AK\wedge X]_{pq}\\
&=[AK \wedge X, AK\wedge X]_{AK,pq}
\end{align*}
The last term is related to \(\textnormal{End}_{AK(n)_{**}}(AK(n)_{**}(X))_{pq}\) via the Universal coefficient spectral sequence(c.f \cite[Prop. 7.7]{DI2}]. The \(E_2\)-term of this spectral sequence is given by \[\textnormal{Ext}_{AK(n)_{**}}(AK(n)_{**}(X),AK(n)_{**}(X))\]
and it converges conditionally to \([AK \wedge X, AK\wedge X]_{AK,pq}\). In particular, if \(AK(n)_{**}(X)\) is free or just projective as an \(AK(n)_{**}\)-module, this spectral sequence collapses at the \(E_2\)-page because it is concentrated in the 0-line, and we get an isomorphism:
\[AK(n)_{pq}(X\wedge DX)\cong \textnormal{End}_{AK(n)_{**}}(AK(n)_{**}(X))_{pq}\]
However, there is no general reason why \(AK(n)_{**}(X)\) should be free or projective. In contrast to this, all graded modules over the graded field \(K(n)_{*}\) are free, and therefore we always have an isomorphism \[K(n)_{**}(X\wedge DX)\cong \textnormal{End}_{K(n)_{*}}(K(n)_{*}(X))\] for all finite topological cell spectra \(X\). As a consequence, instead of working with \(AK(n)_{**}(X\wedge DX)\), we will work directly with \(\textnormal{End}_{AK(n)_{**}}(AK(n)_{**}(X))\) motivically.\\
\end{rem}
\ \\
Every element in \(AK(n)_{**}\) induces a map in \(\text{End}_{AK(n)_{**}}(AK(n)_{**}(X))\) given by multiplication with that element. We will denote this map by the same name as the element. We can now prove the motivic equivalent of asymptotic uniqueness:

\begin{lemma}
\label{Lemma611}
Let \(X\) be a motivic spectrum in \(\mathcal{SH}_{\C}^{qfin}\) or \(\mathcal{SH}_{\C,l}^{\wedge,qfin}\) and \[f: X\rightarrow X\] a motivic \(v_n\)-self map. Then there exist integers \(i\) and \(j\) such that: \[AK(n)_{**}(f^i)=v^j_n\]
\end{lemma}
\begin{proof}
We will use the classical statement for \(v_n^{top}\)-self maps in the topological stable homotopy category. In addition, it is known that for any unit \(u\) in a \(K(n)_*\)-algebra that is finitely generated as a \(K(n)_*\)-module (c.f. \cite[Lemma 3.2]{HS} or \cite[Proof of Lemma 6.1.1]{RAV2}) there is a power of that element such that \(u^i=(v_n^{top})^j\). We will deduce the motivic statement by applying these classical lemmas twice. On the one hand, one can divide out the ideal generated by \(\tau\), which yields a finitely generated \(K(n)_*\)-algebra; on the other hand, one can apply Betti realization.\\
\ \\
Our first step is to show that the map
\[\tau: AK(n)_{**}(X)\rightarrow AK(n)_{**}(X)\]
can not be a unit in \(\text{End}_{AK(n)_{**}}(AK(n)_{**}(X))\):\\
The element \(\tau \in AK(n)_{**}\) is not a unit; if we fix the first degree \(p\) in \(AK(n)_{pq}\) and vary the height \(q\), then there is a maximum height such that \(AK(n)_{pq}=0\) for all larger heights \(q\). If \(\tau\) were a unit, all its powers \(\tau^k\in AK(n)_{0,-k}\) would need to have an inverse \(\tau^{-k}\in AK(n)_{0,k}\) in arbitrarily high weights, which is a contradiction to the previous statement. By the same argument the image of \(\tau\) cannot be a unit in any finitely generated \(AK(n)_{**}\)-module. But \(\text{End}_{AK(n)_{**}}(AK(n)_{**}(X))\) was finitely generated by \ref{FinGen}, so the multiplication-by-\(\tau\)-map cannot be a unit.\\
\ \\
In the second step, we show that the statement is true modulo \(\tau\):\\
The motivic \(v_n\)-self map \(f\) induces an isomorphism in \(AK(n)_{**}\)-homology, i.e. a unit in \(\text{End}_{AK(n)_{**}}(AK(n)_{**}(X))\). In the previous step we showed that \(\tau\) cannot be a unit; this implies that it cannot divide \(AK(n)_{**}(f)\), for if it did, \(\tau\) would also be a unit.\\
Therefore \(AK(n)_{**}(f)\) does not map to zero under the quotient map
\[\text{End}_{AK(n)_{**}}(AK(n)_{**}(X))\rightarrow \text{End}_{AK(n)_{**}/(\tau)}(AK(n)_{**}(X)/(\tau)))\]
and its image \(\overline{AK(n)_{**}(f)}\) is thus a unit in the second ring.\\
If we forget the second bidegree, \(AK(n)_{**}/(\tau)\) is isomorphic to \(K(n)_{*}\), and \(\text{End}_{AK(n)_{**}/(\tau)}(AK(n)_{**}(X)/(\tau))\) is a finitely generated \(K(n)_{*}\)-algebra. In this case we know that there are integers \(i\) and \(j\) such that \(\overline{AK(n)_{**}(f)}^i=(v^{top}_n)^j\). Hence \[AK(n)_{**}(f)^i=v_n^j+\tau \tilde{x}\]
for some element \(\tilde{x}\in AK(n)_{**}(X)\).\\
\ \\
For the last step, suppose now that \(\tilde{x}\) is \(\tau\)-primary torsion. For the fixed prime \(l\) and any \(k\in \N\) we can consider powers \(AK(n)_{**}(f)^{ikl}=v_n^{jkl}+(\tau \tilde{x})^{kl}\). If \(k\) is sufficiently large, the second term vanishes and we are done.\\
Suppose then that \(\tilde{x}\) is not \(\tau\)-primary torsion. Motivic realization induces a map \[\text{End}_{AK(n)_{**}}(AK(n)_{**}(X))\rightarrow \text{End}_{K(n)_{*}}(K(n)_{*}(X))\]
By the classical statement we know that there are integers \(i'\) and \(j'\) such that \(R_{End,X}(f)^{i'}=(v_n^{top})^{j'}\).
Replace \(i,i'\) and \(j,j'\) with their products \(i\cdot i'\) and \(j\cdot j'\) and call the result \(i\) and \(j\) again. Then \(AK(n)_{**}(f)^i=v_n^j+\tau \tilde{x}\) realizes to \(v_n^j\), so \(\tau \tilde{x}\) realizes to \(0\). Because \(\tau\) realizes to 1, \(\tilde{x}\) realizes to \(0\) and by \ref{taurealization} is therefore \(0\) itself. 
\end{proof}

\begin{lemma}
\label{Centrality}
Assume that the motivic nilpotence conjecture holds. Let \(X \in \mathcal{SH}_{\C}^{qfin}\), which implies \(DX \in \mathcal{SH}_{\C}^{qfin}\) by \ref{CellDualizable}.
If \[f: \Sigma^{p,q}X\rightarrow X\] is a motivic \(v_n\)-self map and \[x\in \pi_{p,q}(DX\wedge X)\] is the element corresponding to \(f\) under motivic Spanier Whitehead duality, then there exists an integer \(i \in \N\) such that \(x^i\) is in the center of \(\pi_{p,q}(DX\wedge X)\).
\end{lemma}
\begin{proof}
The proof is essentially similar to \cite[Lemma 3.5]{HS} and \cite[Lemma 6.1.2]{RAV2}, but we will have to use the motivic Nilpotence conjecture at one point.\\
For all \(a\in \pi_{**}(DX\wedge X)\) there is an abstract map of rings
\[ad(a):\pi_{**}(DX\wedge X)\rightarrow\pi_{**}(DX\wedge X)\]
defined by \(ad(a)(b)=ab-ba\), and the element \(a\) is central if and only if \(ad(a)\) is the zero map. This map is realized in homotopy by the composite (here we write \(R\) for \(DX\wedge X\) and \(T\) for the transposition map):
\[S^{p,q}\wedge R \overset{a\wedge id_R}{\rightarrow} R\wedge R \overset{1-T}{\rightarrow} R\wedge R \overset{\mu}{\rightarrow} R\]
We also denote this composite by \(ad(a)\).\\
It now suffices to show that \(ad(x)\) is nilpotent because of the following classical formula (proved in \cite[Lemma 6.1.2]{RAV2}):
\[ad(x^i)(b)=\sum_{j=1}^{i}{{i}\choose{j}}ad^j(x)(b)x^{i-j}\]
If we choose \(i=l^N\) for a sufficiently large \(N\), all summands in this formula vanish either because of the nilpotence of \(ad(x)\) or because the binomial coefficient annihilates \(ad(x)\).\\
Note that \(AK(n)_{**}(DX\wedge X)\) is a finitely generated \(AK(n)_{**}\)-algebra that maps to \(K(n)_{*}(DR(X)\wedge R(X))\) under Betti realization. It follows by the same reasoning as in the proof of Lemma \ref{Lemma611} that a suitable power of \(AK(n)_{**}(x)\) is given by \(v_n^i\) for some \(i\in \N\), which is in the image of \(AK(n)_{**}\) in \(AK(n)_{**}(DX\wedge X)\) and hence central. Replace \(x\) with that power and name it \(x\) again. Then \(AK(n)_{**}(ad(x))\) is zero, so \(ad(x)\) is nilpotent by the nilpotence conjecture. 
\end{proof}

\begin{lemma}
Let \(X\) be a motivic spectrum in \(\mathcal{SH}_{\C,(l)}^{qfin}\) or \(\mathcal{SH}_{\C,l}^{\wedge,qfin}\). Assume that the motivic nilpotence conjecture 	\ref{NilpotenceConjecture} holds. If \(f,g: X\rightarrow X\) are two motivic \(v_n\)-self maps, then there exist integers \(i,j \in \N\) such that \(f^{i}=g^{j}\).
\end{lemma}
\begin{proof}
	This lemma corresponds to \cite[Lemma 3.6]{HS} and \cite[Lemma 6.1.3]{RAV2}. By the previous two lemmas, we can assume that  \(f\) and \(g\), after replacing them with appropiate powers of themselves,  commute with each other in regard to composition, and furthermore that \[AK(n)_{**}(f^{i'}-g^{j'})=0.\]
	Using the nilpotence conjecture, we can conclude that \(f^{i'}-g^{j'}\) is nilpotent. Then \cite[Lemma 3.4]{HS} gives us the desired statement.
\end{proof}

\begin{lemma}
\label{ExtendedUniqueness}
Assume that the motivic nilpotence conjecture \ref{NilpotenceConjecture} holds. If \(f: X\rightarrow X\) and \(g: Y\rightarrow Y\) are two \(v_n\) self maps of \(X\) and \(Y\) and \(h:X\rightarrow Y\) is any map, then there exist integers \(i,j \in \N\) such that \(h\circ f^{i}=g^{l^m}\circ h\).
\end{lemma}
\begin{proof}
The proof is entirely similar to \cite[6.1.4]{RAV2}
\end{proof}

\begin{theorem}
\label{ThickSubcat}
Let \(k=\C\) and \(l\) be an odd prime. Assume that the motivic nilpotence conjecture 	\ref{NilpotenceConjecture} holds. Then the full subcategories of \(\mathcal{SH}_{\C,(l)}^{qfin}\) and  \(\mathcal{SH}_{\C,l}^{\wedge,qfin}\) consisting of spectra admitting motivic \(v_n\)-self maps are thick. 
\end{theorem}
\begin{proof}
First we prove that the category of spectra admitting motivic \(v_n\)-self maps is closed under retracts:\\
Let \(e: X\rightarrow Y\) be a retract with right inverse  \(s: Y \rightarrow X\) and assume that there is a \(v_n\)-self map \(f:X\rightarrow X\).
By \ref{Centrality} a power of \(f\) commutes with \(s\circ e\), so \(e\circ f \circ s\) is a \(v_n\)-self map.\\
Furthermore the category of spectra admitting motivic \(v_n\)-self maps is closed under cofiber sequences:\\
Let \(X\) and \(Y\) be two spectra with motivic \(v_n\)-self maps \(f: \Sigma^{a ,b}X\rightarrow X\) and \(g: \Sigma^{c,d} Y\rightarrow Y\) and let \(h:X\rightarrow Y\) be any map. By \ref{ExtendedUniqueness} we can, after replacing the self maps with suitable powers, assume that \((a,b)=(c,d)\) and \(h\circ f = g\circ h\). Therefore there exists a map \(k:\Sigma^{a,b}C_h \rightarrow C_h\) making the following diagram commute:
\[
\xymatrix{
X \ar[r]^{h}& Y \ar[r] & C_h\\
\Sigma{a,b}X\ar[u]^f \ar[r]^h & \Sigma^{a,b}Y \ar[r] \ar[u]^g & \Sigma^{a,b}C_h \ar@{.>}[u]^k\\
}
\]
It follows by the five lemma and basic facts about triangulated categories that \(k^2\) is a \(v_n\)-self map on \(C_h\) as desired.
\end{proof}

\section{Existence of a self map on \(\X_n\)}
In \cite{HS} Hopkins and Smith used the Adams spectral sequence to prove the existence of a self map on a spectrum \(X_n\) constructed by Smith. In this section we use their proof together with a suitable motivic spectrum \(\X_n\) constructed by Joachimi to show that at least one spectrum in \(\mathcal{SH}_{\C}^{qfin}\) or \(\mathcal{SH}_{\C,l}^{\wedge,qfin}\) actually has a motivic \(v_n\)-self map. The classical proof relies on computing \[K(n)_{p}(X_n\wedge DX_n)\cong \textnormal{End}_{K(n)_{*}}(K(n)_{**}(X_n))_{p}\] via the Adams spectral sequence, so we run into the same kind of problem as in the previous chapter:  Because motivically not all graded modules over \(AK(n)_{**}\) are free, we first have to show that \(AK(n)_{**}(\X_n)\) is in fact free. This also provides us with a Künneth isomorphism for products involving \(AK(n)_{**}(\X_n)\).\\
\ \\
The proof of the existence of a \(v_n\)-self map also relies on the approximation lemma, which relates the cohomology of the Steenrod algebra in certain degrees to the cohomology of certain subalgebras. We need the motivic analogue of this lemma. To this end we need to make two definitions:
\begin{defin}
\begin{enumerate}
\item Let \(X\) be a motivic spectrum. Call \(X\) \emph{\(k\)-bounded below} if \(\pi_{m,n}=0\) for \(m\leq k\). Similarly, call a bigraded module \(M_{m,n}\) over the motivic Steenrod algebra \emph{\(k\)-bounded below} if \(M_{m,n}=0\) for \(m\leq k\).\\
\item A module over the motivic Steenrod algebra has a vanishing line of slope \(m\) and intercept \(b\) if \(\text{Ext}_A^{s,t,u}(M,H\Z/l^{**})=0\) for \(s>m(t-s)+b\). 
\end{enumerate}
\end{defin}
Note that the preceding definition is exactly like the classic one and the weight is not involved.

\begin{defin}
\begin{enumerate}
\item Let \(\beta\) denote the motivic Bockstein homomorphism, and \(Sq^i\) resp. \(P^i\) denote the motivic Square- and Power operations as constructed by Voevodsky in \cite{VOE2}.
If \(l=2\), define \(A_n\) as the subalgebra of the motivic Steenrod algebra generated by \(Sq^1,Sq^2,...,Sq^{2^n}\) over \(H\Z_l^{**}\).\\
If \(l \neq 2\), define \(A_n\) as the subalgebra of the motivic Steenrod algebra generated by \(\beta, P^1,...P^{n-1}\)  for \(n\neq 0\) and by \(\beta\) for \(n=0\).
\item Fix the monomial \(\Z/l\)-basis for the dual motivic Steenrod algebra defined by the elements \(\tau\), \(\xi_i\) and \(\tau_i\) (if \(l\neq 2\)). The elements \(P^s_t\) in the motivic Steenrod algebra are defined as the dual elements to \(\xi^{p^s}_t\), and the elements \(Q_i\) are defined as the dual elements to \(\tau_i\) if \(l\neq2\) and as \(Q_i=P^0_{i+1}\) in the case \(l=2\).
\item Write \(\Lambda(Q_n)\) for the exterior algebra over the ground ring  \(H\Z/l^{**}\) in the generator \(Q_n\). This is a subalgebra of the motivic Steenrod algebra.
\end{enumerate}
\end{defin}
\ \\
We can now prove the motivic analogon to the approximation lemma (c.f. \cite[6.3.2]{RAV2}):
\begin{prop}
Let \(M\) be a bounded below module over the motivic Steenrod algebra such that 
\(\text{Ext}_A^{s,t}(M,H\Z/l^{**})\) has a vanishing line of slope \(m\) and intercept \(b\). \\
For sufficiently large \(N\) the restriction map \[\text{Ext}_A^{s,t}(M,H\Z/l^{**})\rightarrow \text{Ext}_{A_N}^{s,t}(M,H\Z/l^{**})\]
is an isomorphism in degrees \(s\geq m(t-s)+b'\), where \(b'\) can be chosen arbitrarily low for sufficiently large \(N\).
\end{prop}
\begin{proof}
Define \(C\) as the kernel of the surjective map of \(A\)-modules\\ \(A\underset{A_N}{\otimes}M \rightarrow M\). As an \(A_N\)-module \(C\) is given by \(M\otimes \overline{A//A_N}\), where \(A//A_N=A\underset{A_N}{\otimes}\Z/(l)\) and the bar denotes the augmentation ideal.
The motivic Steenrod squares \(Sq^i\) live in bidegrees \((2i,i)\) if \(i\) is even and \((2i+1,i)\) if it is odd and the motivic Power operations \(P^i\) live in bidegrees \((2i(l-1),i(l-1))\). Hence \(\overline{A//A_N}\) will be \(k\)-bounded below, and \(k\) can be chosen arbitrarily high if \(N\) is sufficiently large. Therefore \(C\) has a vanishing line of the same slope as \(M\) and arbitrarily low intercept for sufficiently large \(N\) , cf. \cite{HS}[4.4]. The short exact sequence defining \(C\) and the change-of-rings isomorphism for \(A_N\) and \(A\) provide the following diagram:

\[
\xymatrix{
\text{Ext}^{s-1}_{A}(C,H\Z/l^{**}) \ar[d]\ & \\
\text{Ext}^{s}_{A}(M,H\Z/l^{**})\ar[d]\ar[dr]^{\phi} & \\
\text{Ext}^{s}_{A}(A\underset{A_N}{\otimes}M,H\Z/l^{**})
\ar[d]\ar[r]^{\cong} &
\text{Ext}^{s}_{A_{N}}(M,H\Z/l^{**}) \\
\text{Ext}^{s}_{A}(C,H\Z/l^{**}) &
\\
}
\]
If the upper and lower term in the diagram vanish - which is the case above the vanishing line of \(C\) - the map \(\phi\) is the composite of two isomorphisms and hence an isomorphism itself.
\end{proof}
\ \\
In \cite[Theorem 8.5.12]{JOA} Joachimi defined a motivic cell spectrum \(\X_n\) analogous to the Smith-construction spectrum \(X_n\) in \cite{HS}(see also \cite{RAV2}) by splitting off a wedge summand of a finite cell spectrum via an idempotent. We need some of the details of the construction of \(\X_n\) and its properties for the construction of the \(v_n\)-self map, so we recall and collect all those that are relevant in one place:
\begin{defin}
The spectrum \(\X_n\) is defined as \[\X_n=e_V(\mathbb{B}_{(l)}^{\wedge k_V})=\underset{\rightarrow}{\text{hocolim }}\mathbb{B}_{(l)}^{\wedge k_V}\underset{e_V}{\rightarrow}\mathbb{B}_{(l)}^{\wedge k_V}\underset{e_V}{\rightarrow}...\]
where
\begin{itemize}
\item \(\mathbb{B}_{(l)}\) is a motivic \(l\)-local finite cellular spectrum defined in \cite[8.5]{JOA}, implicitly depending on \(n\).
\item \(V=H\Z/l^{**}(\mathbb{B}_{(l)})=H\Z/l^{**}(a,b)/(a^2,b^{l^n})\), where \(|a|=(1,1)\) and \(|b|=(2,1)\) (\cite[8.5.10]{JOA})
\item \(k_V\) is an integer dependent on \(V\).
\item \(e_V\) is an idempotent of the groupring \(\Z_{(l)}[\Sigma_{k_V}]\), which acts on \(\mathbb{B}_{(l)}^{\wedge k_V}\) by permuting the smashfactors and adding maps.
\item On the level of cohomology, the effect of this idempotent is to split of a free, nonzero \(H\Z/l^{**}\)-submodule of \(V^{\otimes k_V}\). In particular, the motivic cohomology of \(\X_n\) is bounded below as a module over the Steenrod algebra.
\end{itemize}
\end{defin}

Furthermore Joachimi proves the following statements about \(\X_n\):
\begin{theorem}
\label{XnFacts}
\begin{enumerate}
\item \(AK(s)_{**}(\X_n)=0\) for \(s<n\) and \(AK(n)(\X_n)\neq 0\) (\cite[Theorem 8.5.12]{JOA})
\item The operation \(Q_n\) acts trivially on \(H\Z/l^{**}(\mathbb{B}_{(l)})\). This follows for degree reasons from the description of \(H\Z/l^{**}(\mathbb{B}_{(l)}^{\wedge k_V})\) in the previous remark. Since \(H\Z/l^{**}(\X_n)\) is a \(H\Z/l^{**}\)-submodule of this module, \(Q_n\) acts trivially on \(H\Z/l^{**}(\X_n)\).
\item \(R(\X_n)=X_n\). (\cite[8.6]{JOA})
\end{enumerate}
\end{theorem}
\ \\
By \ref{CellDualizable} \(\X_n\) is dualizable, and its  dual is the retract of a finite cell spectrum.
\ \\
Because the spectrum \(\X_n\) is dualizable, it satisfies the expected relation between homology and cohomology once we show that its cohomology is free:
\begin{lemma}
\begin{enumerate}
\item Let \(E\) be a cellular motivic ring spectrum and \(X\) be a dualizable cellular motivic spectrum. If \(E^{**}(X)\) is a free module over the coefficients \(E^{**}\), then \(\text{Hom}_{E^{**}}(E^{**}(X),E^{**})\cong E_{**}(X)\).
\item Let \(X\) be a dualizable cellular motivic spectrum such that
\begin{itemize}
\item \(H\Z/l^{**}(X)\) is free over \(H\Z/l^{**}\)
\item \(Q_n\) acts trivially on \(H\Z/l^{**}(X)\).
\end{itemize}
Then we have an additive bigraded isomorphism \[\text{Ext}^{s,t,u}_{\Lambda(Q_n)}(H\Z/l^{**}(X),H\Z/l^{**})\cong H\Z/l_{**}(X)[v_n]\]
where \(|v_n|=(1,2(l^n-1),l^n-1)\). Here \(s\) is the homological degree and \(t,u\) correspond to the internal bidegree. (The result also holds multiplicatively, but we are not going to need this.)
\end{enumerate}
\end{lemma}
\begin{proof}
\begin{enumerate}
\item This is the content of \cite[8.1.2]{JOA}, using the universal coefficient spectral sequence of \cite[7.7]{DI2} and the fact that this spectral sequence collapses if \(E^{**}(X)\) is free over \(E^{**}\). Note that the cited corollary is stated only for finite cell spectra and the case \(E=H\Z/l\), but the only properties of \(X\) actually used are cellularity and dualizability, and that the proof also works for any cellular motivic ring spectrum \(E\).
\item This is a classical result that can be proven in the following way:\\
Consider the following resolution of free \(\Lambda(Q_n)\)-modules
\[...\overset{\cdot Q_n}\rightarrow \Lambda(Q_n)\overset{\cdot Q_n}\rightarrow \Lambda(Q_n)\overset{\cdot Q_n}\rightarrow \Lambda(Q_n)\overset{\epsilon}\rightarrow H\Z/l^{**}\]
where the last map is the projection \(\epsilon: \Lambda(Q_n)\rightarrow H\Z/l^{**}\) and apply \((-)\underset{H\Z/l^{**}}{\otimes}H\Z/l^{**}(X)\).\\
The resulting long exact sequence
\[...\overset{\cdot Q_n}\rightarrow \Lambda(Q_n)\underset{H\Z/l^{**}}{\otimes}H\Z/l^{**}(X)\overset{\cdot Q_n}\rightarrow \Lambda(Q_n)\underset{H\Z/l^{**}}{\otimes}H\Z/l^{**}(X)\overset{\epsilon}\rightarrow H\Z/l^{**}(X)\]
is a resolution of the \(\Lambda(Q_n)\)-module \(H\Z/l^{**}(X)\). Here we use the assumption that \(Q_n\) acts trivially on this module in the claim that the last map is a map of \(\Lambda(Q_n)\)-modules.\\
Now apply \(\text{Hom}_{\Lambda(Q_n)}((-),H\Z/l^{**})\) and take cohomology. All maps are zero because the target has the trivial \(\Lambda(Q_n)\)-module structure. Using the isomorphism from the previous part, we can rewrite degreewise:
\begin{align*}
&\text{Hom}_{\Lambda(Q_n)}(\Lambda(Q_n)\underset{H\Z/l^{**}}{\otimes}H\Z/l^{**}(X),H\Z/l^{**})\\
&\cong \text{Hom}_{H\Z/l^{**}}(H\Z/l^{**}(X),H\Z/l^{**})\\
&\cong H\Z/l_{**}(X)
\end{align*}
\end{enumerate}
\end{proof}
\ \\
Recall that the coefficient rings of the classical Morava K-theories are graded fields in the sense that all graded modules over it are free. This is not true of the motivic Morava K-theories in general. The algebraic Morava K-theory of the spectrum \(\X_n\) however is free and finitely generated. To see this, we need to go through the steps of its construction.

\begin{prop}
Let \(k=\C\) and \(l\) be an odd prime. Then \(AK(n)_{**}(\X_n)\) is a free, finitely generated \(AK(n)_{**}\)-module.
\end{prop}
\begin{proof}
To prove the statement it suffices to show that 
\begin{itemize}
\item \(AK(n)_{**}(\X_n)\) is a finitely generated \(AK(n)_{**}\)-module
\item \(AK(n)_{**}(\X_n)\) has no \(\tau\)-torsion.
\end{itemize}
We are going to show both claims in three steps: First we compute \(Ak(n)(\mathbb{B}_{(l)})\) using the motivic Adams spectral sequence. We show that it is finitely generated and does not have \(\tau\)-torsion, which implies that \(AK(n)(\mathbb{B}_{(l)})\) is finitely generated and torsionfree. Then we use the Künneth theorem to show the same statement for \(AK(n)(\mathbb{B}_{(l)}^{\wedge k_V})\). Finally we use the definition of the idempotent defining \(\X_n\) to show that \(AK(n)_{**}(\X_n)\) satisfies both claims.\\
We begin with the first step: The motivic Adams spectral sequence for
\(Ak(n)\wedge \mathbb{B}_{(l)}\) converges strongly to \(Ak_{**}(\mathbb{B}_{(l)})\) (\cite[8.3.3]{JOA}). We claim that there are no nontrivial differentials in this spectral sequence.
The \(E_2\)-term of this motivic Adams spectral sequence can be written as \[\text{Ext}_{\Lambda(Q_n)}(H\Z/l^{**}(\mathbb{B}_{(l)}),H\Z/l^{**})\] by change of rings (\cite[8.2.3]{JOA}). 
Recall that \(H\Z/l^{**}(\mathbb{B}_{(l)})=H\Z/l^{**}(a,b)/(a^2,b^{l^n})\). 
The element \(Q_n\) acts trivially on this free and finitely generated \(H\Z/l^{**}\)-module, which implies by the previous lemma \[\text{Ext}^{***}_{\Lambda(Q_n)}(H\Z/(l)^{**}(\mathbb{B}_{(l)})), H\Z/l^{**})\cong H\Z/l_{**}[v_n]\underset{H\Z/(l)_{**}}{\otimes} H\Z/(l)_{**}(\mathbb{B}_{(l)})\]
The right hand side is a tensor product of polynomial algebras, and the position of the polynomial generators and of \(v_n\) in the spectral sequence imply that they cannot support a nontrivial differential at any stage. In the following sketch of the spectral sequence in an abuse of notation \(a\) and \(b\) denote  the dual of the cohomology classes with the same name. Note that the spectral sequence to the right of the depicted area looks very similar to the displayed area - the same elements appear in the same configuration, just multiplied by some power of \(v_n\). In the standard Adams grading the differential \(d_r\) maps one entry to the left and \(r\) entries up. Thus it is clear that no potentially nontrivial differential can have a target different from zero.
\[
\xy
(-7,19)*{\mathrm{s}};(23,-8)*{\mathrm{t-s}};
(-3,3)*{0};(-3,18)*{1};(-3,30)*{...};
(-3,3)*{0};(3,-3)*{0};(9,-3)*{1};
(15,-3)*{2};(21,-3)*{...};(36,-3)*{l^n-1};(54,-3)*{2(l^n-1)};
(3,3)*{a};
(9,3)*{b};
(9,9)*{ab};
(15,3)*{b^2};
(15,9)*{ab^2};
(21,3)*{...};
(36,3)*{b^{l^n-1}};
(36,9)*{ab^{l^n-1}};
(54,18)*{v_n};
(54,24)*{av_n};
\ar@{-}(0,0);(0,40);
\ar@{-}(0,0);(60,0);
\endxy
\]
Therefore \(Ak(n)(\mathbb{B}_{(l)})\) is finitely generated over \(Ak(n)_{**}\) and does not have \(\tau\)-primary torsion. For all cellular spectra \(X\) we have \(AK(n)_{**}(X)\cong v_n^{-1}Ak(n)_{**}(X)\). Therefore \(AK(n)(\mathbb{B}_{(l)})\) is free and finitely generated over \(Ak(n)_{**}\).
\\
The second step is now easy: Since \(AK\) is a cellular spectrum and since we just proved that the cellular spectrum \(\mathbb{B}_{(l)}\) has free \(AK\)-homology over the coefficients, we can apply the Künneth theorem ((\cite{DI2}[Remark 8.7])) and obtain \[AK(n)_{**}(\mathbb{B}_{(l)}^{\wedge k_V})\cong AK(n)_{**}(\mathbb{B}_{(l)})^{\otimes k_V}\]
Therefore also \(AK(n)_{**}(\mathbb{B}_{(l)}^{\wedge k_V})\) is free and finitely generated over the coefficients.\\
For the last step, note that \(AK(n)_{**}(\X_n)\) is a finitely generated \(AK(n)_{**}\)-module as well, since it is a submodule of the finitely generated module \(AK(n)_{**}(\mathbb{B}_{(l)}^{\wedge k_V})\) over the noetherian ring \(AK(n)_{**}\).\\
It remains to show that no torsion occurs. The idempotent \(e_V\in Z_{(l)}[\Sigma_{k_V}]\) acts on \(AK(n)_{**}(\mathbb{B}_{(l)}^{\wedge k_V})\) by permutation of the tensor factors and multiplication by integers. No \(\tau\)-Torsion can occur in \(e_V(AK(n)_{**}(\mathbb{B}_{(l)}^{\wedge k_V}))=AK(n)_{**}(\X_n)\) because the order of an element in a fixed bidegree in \(AK(n)_{**}(\mathbb{B}_{(l)}^{\wedge k_V})\) is the same as that of the \(\tau\)-multiples of that element. Consequently, \(AK(n)_{**}(\X_n)\) is a free \(AK(n)_{**}\)-module.
\end{proof}

\begin{defin}
	Define \(R=D\X_n \wedge \X_n\). It is a quasifinite cell spectrum by definition and by \ref{SWRingSpectrum} it can be endowed with the structure of a motivic homotopy ring spectrum, with unit map \(e: S\rightarrow D\X_n\wedge X_n\) and multiplication map \(\mu: R\wedge R \rightarrow R\).
\end{defin}
\ \\
As a corollary of the preceding proposition we get the following:
\begin{corol}
	\label{Kunneth}
Let \(k=\C\) and \(R=D\X_n \wedge \X_n\). There are Künneth isomorphisms
\begin{enumerate}
\item \[AK(n)^{**}(R)\overset{\cong}{\rightarrow}AK(n)^{**}(D\X_n)\underset{AK(n)^{**}}{\otimes}AK(n)^{**}(\X_n)\]
\item \[AK(n)_{**}(R)\overset{\cong}{\rightarrow}AK(n)_{**}(D\X_n)\underset{AK(n)_{**}}{\otimes}AK(n)_{**}(\X_n)\]
\end{enumerate}
\end{corol}
\begin{proof}
The Milnor short exact sequence for \(\X_n\) and \(AK(n)\)-cohomology is
\[0\rightarrow \underset{\leftarrow}{\text{lim}^1}AK(n)^{*-1,*}(\mathbb{B}_{(l)}^{\wedge k_V})\rightarrow AK(n)^{**}(\X_n) \rightarrow \underset{\leftarrow}{\text{lim }}AK(n)^{**}(\mathbb{B}_{(l)}^{\wedge k_V})\rightarrow 0\]
Because the map \(e_V\) over which the homotopy colimit defining \(\X_n\) is taken is an idempotent, the system \(AK(n)^{**}(\mathbb{B}_{(l)}^{\wedge k_V})\) is Mittag-Leffler, which implies that the \(\underset{\leftarrow}{\text{lim}^1}\)-term vanishes. By the same argument, we have:
\[\underset{\leftarrow}{\text{lim }} AK(n)^{**}(D\X_n \wedge \mathbb{B}_{(l)}^{\wedge k_V})\cong AK(n)^{**}(D\X_n \wedge \X_n)\]
Here the limit is taken over the maps \(id\wedge e_V\).\\
Since \(\mathbb{B}_{(l)}^{\wedge k_V} \) is the \(l\)-localization of a finite cell spectrum and \(AK(n)^{**}(\mathbb{B}_{(l)}^{\wedge k_V})\) is a free module over \(AK(n)^{**}\), we can use the Künneth-isomorphism of Dugger and Isaksen \cite[Remark 8.7]{DI2} to see that \[AK(n)^{**}(D\X_n \wedge \mathbb{B}_{(l)}^{\wedge k_V})\overset{\cong}{\rightarrow}AK(n)^{**}(D\X_n)\underset{AK(n)^{**}}{\otimes}AK(n)^{**}(\mathbb{B}_{(l)}^{\wedge k_V})\]
It remains to rewrite the inverse limit over the right hand side: \(AK(n)^{**}(D\X_n)\) is a free \(AK(n)^{**}\)-module because \(AK(n)_{**}(\X_n)\) is a free \(AK(n)_{**}\)-module, so using the earlier isomorphism we get:
\[ \underset{\leftarrow}{\text{lim }}\big(AK(n)^{**}(D\X_n)\underset{AK(n)^{**}}{\otimes}AK(n)^{**}(\mathbb{B}_{(l)}^{\wedge k_V})\big)\cong AK(n)^{**}(D\X_n)\underset{AK(n)^{**}}{\otimes}AK(n)^{**}(\X_n)\]\\
The Künneth-isomorphism in \(AK(n)\)-homology can either be derived from the one in cohomology or from the Künneth-isomorphism of the \(l\)-local finite cell spectra \(\mathbb{B}_{(l)}^{\wedge k_V}\) and the fact that homology commutes with direct limits.
\end{proof}
\ \\
We also need the following vanishing line:
\begin{lemma}
Let \(l\) be odd and \(R=D\X_n \wedge \X_n\) as before. The \(A^{**}\)-module \[\textnormal{Ext}_{A^{**}}(H\Z/l^{**}(R),H\Z/l^{**})\] has a vanishing line of slope \(1/2(l^n-1)\).
\end{lemma}
\begin{proof}
	Over odd primes, the motivic Steenrod-algebra is just the classical Steenrod algebra (where the generators are understood to live in the appropiate motivic bidegrees) base changed to \(H\Z/l^{**}\). Similarly, \(H\Z/l^{**}(D\X_n \wedge \X_n)\) corresponds to \(H\Z/l_{*}(DX_n\wedge X_n)\) basechanged to \(H\Z/l^{**}\), where the generators are once again understood to live in the appropiate bidegree.\\
	Consequently \(\text{Ext}_{A^{**}}(H\Z/l^{**}(D\X_n \wedge \X_n),H\Z/l^{**})\), which maps to the classical Ext-term \(\text{Ext}_{A^{*}_{\text{top}}}(H\Z/l^{*}(DX_n\wedge X_n),H\Z/l^{*})\), is just that classical Ext-term base changed to \(H\Z/l^{**}\) and in particular does not contain \(\tau\)-torsion. The existence of the vanishing line then follows from the existence of a vanishing line with the same slope in the classical case for the spectrum \(X_n\)(see \cite[6.3.1]{RAV2}).
\end{proof}
\ \\
Furthermore, we need the following duality isomorphisms:
\begin{prop}
Let \(R=D\X_n \wedge \X_n\) as before:
\leavevmode
\begin{enumerate}
\item \(\textnormal{Hom}_{H\Z/l^{**}}(H\Z/l^{**}(R),H\Z/l^{**})\cong H\Z/l_{**}(R)\)
\item \(AK(n)_{**}(D\X_n)\cong AK(n)^{**}(\X_n)\cong \textnormal{Hom}_{AK(n)_{**}}(AK(n)_{**}(\X_n), AK(n)_{**})\)
\item \(AK(n)^{**}(D\X_n)\cong AK(n)_{**}(\X_n)\cong \textnormal{Hom}_{AK(n)^{**}}(AK(n)^{**}(\X_n), AK(n)^{**})\)
\end{enumerate} 
\end{prop}
\begin{proof}
\begin{enumerate}
\item \(R=D\X_n\wedge\X_n\) is a dualizable cell spectrum since \(\X_n\) and \(D\X_n\) are. Therefore we can consider the universal coefficient spectral sequence of \cite[7.7]{DI2}. As explained in \cite[8.1.2]{JOA}, this spectral sequence collapses if \(H\Z/l^{**}(R)\) is free over \(H\Z/l^{**}\). (Note that the cited corollary is stated for finite cell spectra, but the only properties actually used are cellularity and dualizability.) To show the freeness of \(H\Z/l^{**}(R)\) as a \(H\Z/l^{**}\)-module, observe that \(H\Z/l^{**}(\X_n)\) is free by construction (\cite[8.5.3]{JOA}). This implies the existence of a Künneth isomorphism for \(\X_n\), and thus \[H\Z/l^{**}(R)=H\Z/l^{**}(D\X_n)\underset{H\Z/l^{**}}\otimes H\Z/l^{**}(\X_n)\] is free.
\item The first isomorphism follows directly from the canonical bijection. The second isomorphism is proven by the same argument as in the proof of part 1, using the universal coefficient spectral sequence \cite[7.7]{DI2} together with the facts that \(AK\) is a cellular spectrum and that \(AK(n)_{**}(\X_n)\) is free over the coefficients.
\item This is proven just as in part 1 or part 2.
\end{enumerate}
\end{proof}

\begin{corol}
\begin{enumerate}
\item There exists a well defined coevaluation map \[coev: AK_{**}\rightarrow AK_{**}(\X_n)^\vee \underset{AK_{**}}{\otimes} AK_{**}(\X_n)\]
Here \((-)^\vee\) denotes the linear dual \(\textnormal{Hom}_{AK_{**}}(-, AK_{**})\).
It is induced by the map \(T\circ e: S\rightarrow \X_n \wedge D\X_n\), where \(e:S\rightarrow D\X_n \wedge \X_n\) is the unit map of \(R=D\X_n\wedge \X_n\) and \(T\) is the map that transposes the two factors.
\item Under the composition \[AK_{**}\rightarrow AK_{**}(R)\rightarrow \textnormal{Hom}_{AK_{**}}(AK_{**}(\X_n),AK_{**}(\X_n))\] an element \(v\in AK_{**}\) maps to multiplication by that element.
\end{enumerate}
\end{corol}
\begin{proof}
\begin{enumerate}
\item The coevalution map of \ref{SWEQUI}, which is the same  as \(T\circ e\), induces the claimed map in \(AK(n)\)-homology, together with the identification \[AK(n)_{**}(D\X_n)\cong AK^{**}(\X_n)\cong \textnormal{Hom}_{AK_{**}}(AK_{**}(\X_n), AK_{**})\] of the preceding proposition. Because \(AK(n)_{**}(\X_n)\) is a free and finitely generated \(AK(n)_{**}\)-module, there is also an algebraic coevalution defined via choosing a basis as for a vector space, and the two maps coincide since they both satisfy the equivalent of the condition of the first point of \ref{SWEQUI} for projective and finitely generated modules.
\item The element \(1\in AK_{**}\) maps to the coevaluation of \(AK_{**}(\X_n)\) under the first map, using the identification \(AK_{**}(R)\cong AK_{**}(\X_n)^\vee \underset{AK_{**}}{\otimes} AK_{**}(\X_n)\) implied by the Künneth and duality isomorphisms. Hence an element of \(AK_{**}\) maps to that element times the coevaluation. The coevaluation maps to the identity under the second map. Consequently an element in \(AK_{**}\) times the coevaluation maps to multiplication by that element.
\end{enumerate}
\end{proof}
\ \\
We now have all the ingredients to use the classical proof in the motivic setting (\cite[Theorem 4.12]{HS}, see also \cite[6.3]{RAV2}): 
\begin{theorem}
\label{SelfMapExample}
Let \(k=\C\) and \(l\) be an odd prime. The spectrum \(\mathbb{X}_n\) has a motivic \(v_n\) self-map \(f\) satisfying \[AK(m)_*f=\delta_{mn}v_n^{p^{N_m}}\]
for a sufficiently large integer \(N_m\).
\end{theorem}
\begin{proof}
The aim is to construct a permanent cycle
\[v\in \text{Ext}_{A_{**}}(H\Z/l_{*}(R),H\Z/l_{*})\]
that maps to a power of \(v_n\) in \(Ak(n)_{**}(R)\) and to a nilpotent element in \(Ak(m)_{**}(R)\) if \(m\neq n\). The diagram below will specify the meaning of "maps". Under motivic Spanier Whitehead duality such a class corresponds to a self-map of the described form on \(\X_n\).
\\
The cohomology of the point, \(H\Z/l^{**}\), is concentrated in simplicial degree 0. Therefore the operations \(Q_n\) act trivially on this module over the motivic Steenrod algebra. They act trivially on \(H\Z/l^{**}(R)\) since they act trivially on \(H^{**}(\X_n)\). If we write \(P(v_n)\) for the polynomial algebra in one generator with respect to the base ring \(H\Z/(l)_{**}\), this provides us with the following isomorphisms of trigraded algebras:
\begin{align}
\text{Ext}^{***}_{\Lambda(Q_n)}(H\Z/l^{**},H\Z/l^{**})&\xrightarrow{\cong} P(v_n)\otimes H\Z/l_{**}\\
\text{Ext}^{***}_{\Lambda(Q_n)}(H\Z/l^{**}(R),H\Z/l^{**})&\xrightarrow{\cong} P(v_n)\otimes H\Z/l_{**}(R)
\end{align}
Here \(v_n\) has homological degree 1 and internal bidegree \((2(l^n-1,l^n-1)\).\\
\ \\
Together with the change-of-rings morphisms related to the subalgebras \(\Lambda(Q_n)\) and \(A_N\) these fit into the following diagram:
\[
\xymatrix{
\text{Ext}_{A^{**}}(H\Z/l^{**},H\Z/l^{**}) \ar[r]^{i}\ar[d]^{\phi} & \text{Ext}_{A^{**}}(H\Z/l^{**}(R),H\Z/l^{**}) \ar[d]^{\phi}
\\
\text{Ext}_{A_N^{**}}(H\Z/l^{**},H\Z/l^{**}) \ar[r]^{i} \ar[d]^{\lambda} &
\text{Ext}_{A_N^{**}}(H\Z/l^{**}(R),H\Z/l^{**}) \ar[d]^{\lambda} &
\\
\text{Ext}_{\Lambda(Q_n)}(H\Z/l^{**},H\Z/l^{**}) \ar[r]^{i} \ar[d]^{\cong (1)} &
\text{Ext}_{\Lambda(Q_n)}(H\Z/l^{**}(R),H\Z/l^{**}) \ar[d]^{\cong (2)} &
\\
P(v_n)\underset{H\Z/(l)_{**}}{\otimes} H\Z/l_{**}  \ar[r]^{i} \ar[d]&
P(v_n)\underset{H\Z/(l)_{**}}{\otimes} H\Z/l_{**}(R) \ar[d]&
\\
Ak(n)_{**} \ar[r]^{i} &
Ak(n)_{**}(R)&
\\
}
\]
\ \\
Step 1: Consider the element \(\widetilde{v_n}\in \text{Ext}_{\Lambda(Q_n)}(H\Z/l_{*}(R),H\Z/l_{*})\) that corresponds to \(v_n\otimes 1\in P(v_n)\otimes H_{**}(R)\) under the isomorphism (2).
\begin{prop} \(\forall N\geq n\) there is an integer \(t>0\) and an element \(x \in \text{Ext}_{A_{N,**}}(H\Z/l_{*},H\Z/l_{*})\) such that \(\lambda(x)=v_n^t\). The image of \(x\) under \(i\) is central in \(\text{Ext}_{A_{N,**}}(H\Z/l_{*}(R),H\Z/l_{*})\), where central is meant in respect to graded commutativity in the first, but not in the second bidegree.
\end{prop}
\begin{proof}
This statement is a corollary of \cite[Theorem 4.12]{HS}. Since the motivic cohomology of the point \(H\Z/(l)_{**}=\Z/(l)[\tau]\) is concentrated in simplicial degree 0, the action of the motivic Steenrod algebra is trivial on this module. Hence we can basechange the statement of the cited theorem to \(\Z/(l)[\tau]\).
\end{proof}
\ \\
Step 2: The module \(\text{Ext}_{A^{**}}(H\Z/l^{**}(R),H\Z/l^{**})\) has a vanishing line of slope \(1/(2^l-2)\) and a fixed intercept \(b\). By the motivic approximation lemma, the morphism \[\phi:\text{Ext}_{A^{**}}(H\Z/l^{*}(R),H\Z/l^{**})\rightarrow \text{Ext}_{A_N^{**}}(H\Z/l^{**}(R),H\Z/l^{**})\]
is an isomorphism above a line with slope \(1/2(l^n-1)\) and arbitrarily low intercept for sufficiently large \(N\). Since the element \(x\)(and therefore also \(i(x)\)) has tridegree (\(t,2(l^n-1),(l^n-1)\)), it lies above that line for a sufficiently large choice of \(N\). Define \(y\in \text{Ext}_{A^{**}}(H\Z/l^{**}(R),H\Z/l^{**})\) as the preimage of \(i(x)\) under \(\phi\). Since \(i(x)\) is central (in the graded sense with respect to the first bidegree but not with respect to the second)  in \(\text{Ext}_{A_N^{**}}(H\Z/l^{**}(R),H\Z/l^{**})\), it commutes with all elements in the image of \(\phi\), in particular with all elements above the line defined by the approximation lemma.\\
\ \\
Step 3: The element \(y\) and its powers, as well as the images of \(y\) and its powers under the differentials of the motivic Adams spectral sequence all satisfy the requirement of the last statement, so they commute with each other.  By induction we can assume that a power \(\tilde{y}\) of \(y\) survives up to the \(r\)th page. We wish to show that \(\tilde{y}^l\) is a \(r\)-cycle, i.e. \(d_r(\tilde{y}^l)=0\). This is true since \(d_r(\tilde{y}^l)=l\cdot \tilde{y}^{l-1}d_r(\tilde{y})=0\). After a finite number of pages, the differential will point in the area of the spectral sequence above the vanishing line, and we can stop the process. We end with a power \(\tilde{y}\) of \(y\) that is a permanent cycle in the motivic Adams spectral sequence and hence represents an element of \(Ak(n)_{**}(R)\).\\
\ \\
Step 4: The permanent cycle \(\tilde{y}\) represents an element \(\bar{y}\in \pi_{**}(R)\). Choose \(m\) such that \(v_n^m\) has the same degree as \(\bar{y}\). By the exact same arguments as in \cite{HS} we can choose a power of \(\bar{y}\) such that \(Ak(n)_{**}(\bar{y}^g)=v_n^{gm}\) and define \(f\) as the map corresponding to that power of \(\bar{y}\) under motivic Spanier Whitehead duality.\\
\ \\
Step 5: For \(m \neq n\) it follows just as in the topological case that the image of \(v\) in \(AK(m)_{**}\) is nilpotent either for trivial reasons (\(m<n\)) or because of a vanishing line with tighter slope in the Adams spectral sequence computing \(Ak(m)_{**}(R)\) (\(m>n\)).
\end{proof}

\section{The relation of \(\mathcal{C}_\eta\) and \(\mathcal{C}_{AK(n)}\)}
As a corollary of the Künneth isomorphism, we can settle one of the open conjectures in Ruth Joachimis dissertation \cite[Conjecture 7.1.7.3]{JOA} which concerns the relation of the thick ideal \(\text{thickid}(C_\eta)\) generated by the cone of the motivic Hopf map \(C_\eta\) and the thick ideals \(\mathcal{C}_{AK(n)}\) characterized by the vanishing of motivic Morava K-theory.

\begin{lemma}
	Let \(m\in \N\) be any integer. Then the coefficients of the cone \(C_\eta\) of \(\eta: \Sigma^{1,1}S\rightarrow S\) in \(AK(m)_{**}\)-homology are given by: \[AK(m)_{**}(C_\eta)\cong AK(m)_{**}\oplus AK(m)_{*-2,*-1}\]
	In particular, they are free over \(AK(m)_{**}\).
\end{lemma}
\begin{proof}
	The long exact sequence induced by the cofiber sequence \[S^{1,1}\rightarrow S^{0,0}\rightarrow C_\eta \rightarrow S^{2,1}\] defining \(C_\eta\) splits into short exact sequences \[0 \rightarrow AK(m)_{**}\rightarrow AK(m)_{**}(C_\eta)\rightarrow AK(m)_{*-2,*-1}\rightarrow 0\]
	because \(\eta\) induces the zero map in \(AK(m)_{**}\)-homology. The sequence splits because the outer terms are free \(AK(m)_{**}\)-modules, yielding the result.
\end{proof}

\begin{corol}
	Let \(m\in \N\). In the case \(m<n\) we have
	\[
	AK(m)_{**}(C_\eta\wedge \X_n)\cong0
	\]
	and in the case \(m=n\) we have:
	\[AK(n)_{**}(C_\eta\wedge \X_n)\cong AK(n)_{**}(C_\eta)\underset{AK(n)_{**}}{\otimes} AK(n)_{**}(\X_n)\neq 0\] 
\end{corol}
\begin{proof}
	By the preceding lemma the finite cell spectrum \(C_\eta\) has free \(AK(m)\)-homology and thus satisfies the requirements of the Künneth formula \cite[Remark 8.7]{DI2}.\\
	Application of the Künneth formula yields:
	\[AK(m)_{**}(C_\eta \wedge \X_n)\cong AK(m)_{**}(C_\eta)\underset{AK(m)_{**}}{\otimes} AK(m)_{**}(\X_n)\]
	If \(m<n\) the factor \(AK(m)_{**}(X)=0\) vanishes by \ref{XnFacts}. This implies the first part of the statement. If \(m=n\) the result contains \[AK(n)_{**}(\X_n)\underset{AK(n)_{**}}{\otimes}AK(n)_{**}=AK(n)_{**}(\X_n)\neq 0\]
	as a direct summand, so \(AK(n)_{**}(C_\eta\wedge \X_n)\) cannot vanish.
\end{proof}

\begin{prop}
	The spectrum \(\X_{n+1}\) is contained in the intersection of thick ideals \(\text{thickid}(C_\eta)\cap \mathcal{C}_{AK(n)}\), but not in \(\text{thickid}(C_\eta)\cap \mathcal{C}_{AK(n+1)}\). In particular, these intersections are nonzero and distinct for all \(n\in \N\).
\end{prop}
\begin{proof}
	Clearly \(C_\eta \wedge \X_{n+1}\) is in the thick ideal generated by \(C_\eta\). The preceding corollary tells us on the one hand that \(C_\eta \wedge \X_{n+1} \in \mathcal{C}_{AK(n)}\), and on the other hand that \(C_\eta \wedge \X_{n+1} \notin \mathcal{C}_{AK(n+1)}\).
\end{proof}

\section{A counterexample to a statement about thick subcategories in \cite{JOA}}
In this section we construct a counterexample to the inclusion \[\text{thickid}(c\mathcal{C}_2) \subset \mathcal{C}_{AK(1)}\] claimed in \cite[Chapter 9, last section]{JOA}, based on an error in \cite[Proposition 8.7.3]{JOA}.\\
Let \(l\) be an odd prime, and consider the topological mod-\(l\) Moore spectrum \(S/l \in \mathcal{SH}\). We can easily compute its \(K(1)\)-homology:
\begin{lemma}
\(K(1)_*(S/l)\cong K(1)_*\oplus K(1)_{*-1}\)
\end{lemma}
\begin{proof}
The Moore spectrum is defined via the cofiber sequence \(S\overset{\cdot l}\rightarrow S \rightarrow S/l\) and the map induced by \(l\) is trivial in \(K(1)\)-homology. Therefore the long exact sequence in \(K(1)\)-homology induced by this cofiber sequence splits up into short exact sequences, and these short exact sequences split because all graded \(K(1)\)-modules are free.
\end{proof}
\ \\
In \cite{ADA} Adams proved the existence of a non-nilpotent self map
\[v: \Sigma^{2l-2}S/l\rightarrow S/l\]
on the Moore spectrum which induces an isomorphism in \(K(1)\)-homology; namely multiplication by the invertible element \(v_1^{top}\). Consequently, the \(K(1)\)-homology of the cone \(C_v\) vanishes: \(K(1)_*(C_v)=0\), or equivalently \(C_v \in \mathcal{C}_2\).\\
 \\
Applying the constant simplicial presheaf functor \(c\) to the construction gives us the cofiber sequence \[\Sigma^{2l-2,0}S/l\overset{cv}\longrightarrow S/l\rightarrow C_{cv}\] in \(\mathcal{SH}_\C\). The cone \(C_{cv}\) of \(cv\) is equivalent to \(c(C_v)\) because \(c\) is a triangulated functor, and the Moore spectrum is mapped to the Moore spectrum (\(cS/l=S/l\) because \(cl=l\).) We can compute the \(AK(1)\)-homology of the mod-\(l\)-Moore spectrum using the same argument as in the topological case:
\[AK(1)_{**}(S/l)\cong AK(1)_{**}\oplus AK(1)_{*-1,*}\]
However, the algebraic Morava K-theory of \(C_{cv}\) does not vanish:
\begin{lemma}
	\label{NonVanishingMorava}
	\(AK(1)(C_{cv}))\cong AK(1)_{**}(S/l)/(\tau^{l-1})\neq 0\)
\end{lemma}
\begin{proof}
	The cofiber sequence \(S/l\overset{cv}\rightarrow S/l \rightarrow C_{cv}\) induces a long exact sequence in \(AK(1)\)-homology:
	\[
	...\rightarrow AK(1)_{p+(2l-2),q}(S/l) \overset{AK(1)_{**}(cv)}\longrightarrow AK(1)_{pq}(S/l)\rightarrow AK(1)_{pq}(C_{cv})\rightarrow ...
	\]
The map \(AK(1)_{**}(cv)\) must be given by multiplication with \(\tau^{l-1}v_1\), because Betti realization maps \(AK(1)_{**}(cv)\) to multiplication with \(v_1^{top}\) and there is only one map realizing to this in the appropiate bidegree. This map is injective but, unlike the topological case, no longer an isomorphism. Hence the long exact sequence splits into short exact sequences
\[0 \rightarrow AK(1)_{pq}(cS/l) \overset{\cdot \tau^{l-1}v_1}\longrightarrow AK(1)_{pq}(cS/l)\rightarrow AK(1)_{pq}(C_{cv}) \rightarrow 0\]
and because \(v_1\) is invertible, the last term is isomorphic to \(AK(1)_{**}(S/l)/(\tau^{l-1})\).
\end{proof}
\ \\
Because \(\mathcal{C}_{AK(1)}\) was defined by the vanishing of \(AK(1)\)-homology and \(AK(1)_{**}(C_{cv})\neq 0\) does not vanish, we have \(C_{cv}\notin \mathcal{C}_{AK(1)}\). On the other hand, we have shown that \(C_v\in \mathcal{C}_2\). Because \(R(C_{cv})=C_v\), this implies \(C_{cv}\in R^{-1}(\mathcal{C}_2)\). Therefore we can conclude the following corollary from the preceding lemma:
\begin{corol}
	The inclusion \[\mathcal{C}_{AK(1)}\subsetneq R^{-1}(\mathcal{C}_2)\] is proper.
\end{corol}
\ \\
Furthermore we have \(C_{cv}=cC_v\in \text{thickid}(c\mathcal{C}_2)\).  Therefore \(cC_v\) is our desired counterexample and proves:
\begin{prop}
\(\text{thickid}(c\mathcal{C}_2) \not\subset \mathcal{C}_{AK(1)}\)
\end{prop}

\begin{rem}The mistake on which the incorrect assertion is based occurs in \cite[Proposition 8.7.3]{JOA}. This proposition states that for a finite topological CW spectrum \(Y\), \(AK(n)^{**}(cY)=0\) if and only if \(K(n)_{*}(Y)=0\). In the proof of this proposition Joachimi shows that the differentials in the motivic Atiyah-Hirzebruch spectral sequence are determined by the differentials of the topological Atiyah-Hirzebruch spectral sequence, and that the \(E_2\)-page of the motivic spectral sequence is given by adjoining a generator \(\tau\) to each entry in the topological spectral sequence, where all entries are generated in motivic weight 0. The problem that now occurs is that the differentials in the motivic spectral sequence do not preserve the weight, but lower it. Hence a nontrivial differential can generate \(\tau\)-primary torsion in the spectral sequence. The above example shows that this in fact happens.
\end{rem}
\ \\
This argument can in fact be made for any topological spectrum \(X\in \mathcal{C}_{n+1}\setminus\mathcal{C}_{n+2}\). Any such spectrum has nontrivial \(
K(n)\)-homology and a self map \(v: \Sigma^m X\rightarrow X\) that induces multiplication by some power of \(v_n^{top}\). We know by \ref{taurealization} that the map \(AK(n)_{**}(cX)\rightarrow K(n)_*(X)\) induced by Betti realization is surjective and its kernel is exactly the \(\tau\)-primary torsion elements. In particular we know that \(AK(n)_{**}(cX)\neq 0\) , and the self map provides us with a motivic map \(cv\). This map induces multiplication by the same power of \(\tau^{l-1}v_n\) in \(AK(n)\)-homology -  up to a possible error term, which has to be \(\tau\)-primary torsion. We can eliminate this error term by taking sufficiently large \(l\)-fold powers of this map. We end up with a \(v_n^{top}\)-self map \(v'\) of \(X\) whose image \(cv'\) under the constant simplicial presheaf funtor \(c\) induces multiplication by some power of \(\tau^{l-1}v_n\) in \(AK(n)\)-homology. In particular, its cone has nonvanishing \(AK(n)\)-homology by the same argument as for our earlier counterexample and thus proves:
\begin{prop}
\(\text{thickid}(c\mathcal{C}_{n+1}) \not\subset \mathcal{C}_{AK(n)}\)
\end{prop}
Just as before, this also proves:
\begin{corol} The inclusion \[\mathcal{C}_{AK(n)}\subsetneq R^{-1}(\mathcal{C}_{n+1})\]
is proper.
\end{corol}

\appendix

\ \\
\textit{Sven-Torben Stahn\\
Fachgruppe Mathematik und Informatik\\
Bergische Universität Wuppertal\\}
SvenTorbenStahn@gmail.com
\end{document}